\documentclass[12pt]{article}
\usepackage[utf8]{inputenc}
\usepackage{csquotes}
\usepackage{amsmath}
\usepackage{graphicx}
\usepackage[title]{appendix}
\usepackage[colorinlistoftodos]{todonotes}
\usepackage[
    colorlinks=true,
    linkcolor=black,
    urlcolor=blue
]{hyperref}
\usepackage{eurosym}
\usepackage{mathtools}
\usepackage{amsfonts,amssymb,amsthm}
\usepackage{commath}
\usepackage{indentfirst}
\usepackage{scrextend}
\usepackage{listings}
\usepackage{float}
\usepackage{caption}
\usepackage{upgreek}
\usepackage{subfigure}

\usepackage{url}
\hypersetup{
    colorlinks = false,
    linkbordercolor = {white}
}
\usepackage{comment}
\usepackage{fancyhdr}
\usepackage{authblk}
\usepackage{multirow}
\usepackage{array}
\usepackage{booktabs}
\usepackage{color,colortbl}
\usepackage{mathtools}
\usepackage{cleveref}
\usepackage{enumitem}

\usepackage[many]{tcolorbox}    	

\definecolor{main}{HTML}{5989cf}    
\definecolor{sub}{HTML}{cde4ff}     

\tcbset{
    sharp corners,
    colback = white,
    before skip = 0.2cm,    
    after skip = 0.5cm      
}                           

\newtcolorbox{boxD}{
    colback = sub, 
    colframe = main, 
    boxrule = 0pt, 
    toprule = 3pt, 
    bottomrule = 3pt 
}

\usepackage[
  natbib=true,
  style=numeric-comp,
  sorting=none,
  giveninits=true,
  maxbibnames=99,
  doi=true,
  isbn=false,
  url=false,
  eprint=false
]{biblatex}
\AtEveryBibitem{\clearfield{month}}

\usepackage[margin=1in]{geometry}

\newtheorem{theorem}{Theorem}            
\newtheorem{lemma}[theorem]{Lemma}       
\newtheorem{proposition}[theorem]{Proposition} 
\newtheorem{remark}[theorem]{Remark}
\newtheorem{definition}[theorem]{Definition}

\begin{document}

\title{Backward bifurcations in spatial replicator models: when invasion criteria fail to predict coexistence}

\author[1]{Tomás Freire}
\author[1]{Erida Gjini}
\author[2]{Sten Madec}

\affil[1]{Center for Computational and Stochastic Mathematics, Instituto Superior Tecnico, Lisbon, Portugal}
\affil[2]{Institut Denis Poisson, University of Tours, Tours, France}
\date{
\small tomas.freire@tecnico.ulisboa.pt \quad erida.gjini@tecnico.ulisboa.pt \quad sten.madec@univ-tours.fr}

\maketitle

\tableofcontents

\newpage
\begin{abstract}
The replicator equation is a central framework for studying frequency–dependent selection in ecology and evolutionary game theory. In well-mixed populations, the long-term outcome of a two-species system is determined by the signs of the pairwise invasion fitnesses, leading to dominance, coexistence, or bistability. However, many ecological systems are spatially structured, with environmental heterogeneity and dispersal shaping local interactions. How these spatial effects modify the classical replicator regimes remains incompletely understood.

In this work, we study a spatially heterogeneous extension of the two-species replicator equation in which pairwise invasion fitnesses vary across space, and frequencies evolve under diffusion and advection. In this setting, the classical invasion fitnesses are replaced by spatial invasion rates given by the principal eigenvalues of the associated linearized operators.

Using bifurcation theory, we show that these spatial invasion rates do not fully determine the qualitative dynamics of the system. In particular, spatial heterogeneity can induce a backward bifurcation, generating stable coexistence states even in parameter regimes where the well-mixed replicator predicts competitive exclusion. We derive an explicit local condition for this mechanism, characterizing when such coexistence states arise.

These results show that spatial structure can fundamentally alter classical replicator dynamics and provide a concrete mechanism through which coexistence may emerge.
\end{abstract}

\section{Introduction}

Understanding when competition leads to coexistence and when it leads to exclusion is a central problem in theoretical ecology.
Classical models, from the Lotka-Volterra framework to replicator dynamics, often predict that coexistence becomes increasingly fragile as the number of interacting species grows. One of the most influential contributions to this topic is Robert May's 1972 paper \cite{MAY1972}, which highlighted a striking paradox: large and complex ecological systems are predicted to be unstable, despite the widespread empirical observation that natural ecosystems can persist for long periods. This apparent contradiction has motivated decades of research aimed at identifying mechanisms that promote stability and coexistence in ecological communities \cite{McCann2000,Chesson2000,Allesina2012}.

Several mechanisms have been proposed to reconcile theory with observation, including stochasticity, interaction network structure, alternative notions of stability, and spatial structure. Among these, spatial heterogeneity can be especially important \cite{Tilman1994,Levin1992}. In many ecological systems, individuals interact locally and experience heterogeneous environments, so that dispersal and environmental variation strongly influence competitive outcomes. Even simple population models can change their qualitative behavior once spatial processes are introduced. A classical example is the Kierstead-Slobodkin-Skellam model, which shows that a population subject to diffusion can only persist if the habitat patch exceeds a critical size determined by dispersal and growth rates \cite{SKELLAM1951,Kierstead1953}. Subsequent theoretical work has extended this framework to heterogeneous environments, showing that spatial variation in growth rates or dispersal can shift persistence thresholds through its effect on the principal eigenvalue of the associated growth-diffusion operator \cite{Cantrell1991,Dockery1998,CANTRELL2001}.

More broadly, spatial dynamics have been studied in detail across many areas of ecology, evolution, and microbiology. In reaction-diffusion and reaction-advection-diffusion models, movement through heterogeneous environments can change persistence thresholds, alter competitive outcomes, and generate coexistence mechanisms that are absent from the corresponding well-mixed systems \cite{Cosner2014}. In spatial Lotka-Volterra and competition models, heterogeneity and dispersal can modify the stability of coexistence states, change the outcome of competition, and select for different movement strategies \cite{Dockery1998,CantrellCosnerLou2010}. In ecological and metacommunity models, spatial heterogeneity and dispersal structure can promote diversity by creating refuges, source-sink effects, or local differences in competitive conditions \cite{Amarasekare2001,Mouquet2003,Leibold2004,Snyder2004}. Similar ideas have become increasingly important in microbial ecology, where spatial structure, short-range interactions, nutrient gradients, and dispersal can strongly shape community assembly and evolution \cite{Custer2022,Bcker2026}. For example, \cite{Ghosh2022} studied spatial models of luminal growth in the human gut microbiota and showed that spatially structured growth and flow can generate effective evolutionary forces that are not apparent in a well-mixed description. These examples illustrate that spatial structure is not only a technical extension of classical models, but can fundamentally change the ecological and evolutionary forces experienced by competing types.

Spatial extensions of evolutionary game dynamics and replicator equations have also been studied in several settings. Early work on spatial evolutionary games showed that local interactions and spatial structure can generate pattern formation, spatial chaos, and changes in evolutionary stability \cite{Vickers1989,Nowak1992,Durrett1994,Nanda2017}. Replicator-diffusion and finite-population spatial replicator models have also been used to study travelling waves and pattern formation in evolutionary games \cite{Vickers1989,Griffin2021}. Related spatial replicator-type equations also arise in epidemiology, where structured coinfection models can lead, under suitable assumptions, to equations for strain frequencies \cite{Le2023}. Recent public goods game models further show how environmental feedbacks, variable population size, and diffusion can shape cooperation and generate complex spatio-temporal dynamics \cite{chengetal2024,wang2026bifurcations}.

Despite these connections, spatially heterogeneous replicator equations are relatively uncommon. One reason is that the passage from abundance-based spatial models to frequency-based equations is not straightforward. In a non-spatial generalized Lotka-Volterra system, the dynamics can often be rewritten in terms of relative frequencies, leading naturally to a replicator equation with payoff matrix \(\Lambda\). In space, however, the total population density generally varies with location. As a result, transforming species abundances \(n_i(x,t)\) into local frequencies \(z_i(x,t)=n_i(x,t)/N(x,t)\), with \(N=\sum_i n_i\), does not in general produce a replicator system with migration. As a consequence, early spatial replicator formulations often retained information about both species abundances \(n_i\) and relative frequencies \(z_i\), rather than closing the model at the frequency level alone \cite{Vickers1989,Vickers1991}. These restrictions can be relaxed by coupling the frequency dynamics to a separate equation for total biomass or population density, following \cite{Durrett1994} and later extensions \cite{Griffin2021,Griffin2024}.

A related conceptual issue concerns the interpretation of payoffs in space. In the classical replicator equation, the quadratic term \(z^T\Lambda z\) represents the average payoff in the population. A spatial version with a local term \(z(x)^T\Lambda(x)z(x)\) assumes that the payoff experienced at position \(x\) depends on the frequency composition at that same position. This is natural when interactions are local, but less immediate when payoffs are mediated by interactions over larger spatial scales. Alternative formulations therefore define payoffs through spatial integrals, with dispersal entering either inside or outside the payoff structure \cite{Bratus2014}.

A concrete setting in which spatial replicator-type equations arise is provided by structured multi-strain epidemiological SIS models with co-colonization \citep{Madec2020,le2023quasi,Le2023}. In these models, suitable reductions based on strain similarity can lead to equations for the local frequencies of strains, with interaction terms and movement acting directly on those frequencies. This gives one biological motivation for the class of spatial replicator equations considered here.

At the same time, the spatial replicator equation can also be studied as a minimal model for frequency-dependent competition in heterogeneous environments, building on the classical replicator framework \cite{PAGE2002,Bomze1995,Hofbauer1998}. Its main advantage in the present setting is dimensional reduction. An \(N\)-species system can be written in terms of \(N-1\) independent frequencies. This reduction is especially useful in the \(N=2\) species case, where the dynamics are described by a single scalar equation for the frequency of one species. The resulting scalar model provides a useful baseline for studying how spatial heterogeneity and dispersal modify classical coexistence criteria and invasion-based predictions, while avoiding the additional complexity of a coupled system of PDEs.

Starting from the scalar replicator equation for the frequency of species~1, we allow the pairwise interaction terms $\lambda_i^j$ to depend on location and add movement through diffusion and advection. This leads to the spatial replicator equation
\begin{equation*}
\frac{\partial z}{\partial t}=
z(1-z)\bigl(\lambda_1^2(x)-(\lambda_1^2(x)+\lambda_2^1(x))z\bigr)
+
\vec{\nu}(x)\cdot\nabla z
+
\nabla\cdot\bigl(D(x)\nabla z\bigr),
\end{equation*}
where $z=z(x,t)$ denotes the local frequency of species~1. 

Although the resulting model remains minimal, it captures key mechanisms through which spatial heterogeneity can influence frequency-dependent competition. Our goal is to isolate the role of spatial structure and determine how it modifies the classical two-species replicator regimes. We show that spatial heterogeneity can generate a backward bifurcation from a boundary equilibrium, producing stable coexistence states in parameter regimes where the corresponding homogeneous model predicts exclusion. In particular, coexistence may occur even when one of the spatial invasion fitnesses is negative and the other positive. The bifurcation analysis identifies an explicit condition for this mechanism, expressed through spatial averages weighted by the principal eigenfunctions of the linearized operator and its adjoint. This provides a concrete link between environmental heterogeneity, dispersal, and the limits of sign-based invasion criteria.

The remainder of the paper is organized as follows. Section~2 recalls the classical replicator dynamics and introduces the spatially heterogeneous reaction-advection-diffusion formulation studied in this work. We then define spatial invasion fitness through the principal eigenvalues of the linearized operators. Section~3 analyzes the stability regimes in the $(\rho_1^2,\rho_2^1)$ plane. Section~4 develops the bifurcation analysis leading to the backward bifurcation mechanism. Section~5 presents biologically relevant examples illustrating these effects, and Section~6 concludes with a discussion of the ecological implications of the results.

\section{Replicator dynamics}
\subsection{Non-spatial replicator dynamics}

To establish a baseline for comparison, we briefly recall the classical replicator dynamics describing frequency-dependent competition in well-mixed populations. These equations have been widely studied due to their broad applicability in evolutionary game theory, ecology, epidemiology, and microbiology \cite{Hofbauer1998,Cressman2014}. The instance of the replicator equation that we consider is written in terms of an invasion fitness matrix $\Lambda=(\lambda_i^j)_{1\leq i,j \leq N}$, where each entry $\lambda_i^j$ represents the pairwise invasion fitness of species $i$ in an environment set by species $j$ \cite{Geritz1998}; that is, $\lambda_i^j$ is the initial growth rate of species $i$ when it is rare in an environment set by $j$. This framework tracks the evolution of the relative frequency of each species and is described by the following $N$-dimensional system of equations:

\begin{equation}
\label{eq:replicatorODE}
\frac{d z_i}{dt}
=
z_i \left(\sum_{j \neq i} \lambda_i^j z_j
-
\sum_{1\leq k<j \leq N} ( \lambda_j^k+\lambda_k^j) z_j z_k \right),
\quad i=1,\dots,N.
\end{equation}

In this model, the linear term $\sum_{j \neq i} \lambda_i^j z_j$ can be interpreted as the context-dependent invasion fitness of species $i$, while the quadratic term represents the resistance of the system to invasion.

In the two-species case ($N=2$), the system reduces to a single equation for the frequency of species~1, denoted $z=z_1$, with $z_2=1-z$:

\begin{equation}
\label{eq:replicator2species}
\frac{d z}{dt}
=
z(1-z)\bigl(\lambda_1^2-(\lambda_1^2+\lambda_2^1)z\bigr).
\end{equation}

Equation~\eqref{eq:replicator2species} has up to three non-negative equilibria:
\[
z_0^*=0,\qquad
z_1^*=1,\qquad
z_{eq}^*=\frac{\lambda_1^2}{\lambda_1^2+\lambda_2^1}.
\]
Their stability depends on the signs of the invasion fitness $\lambda_i^j$, leading to the classical four regimes: dominance of species~1, dominance of species~2, coexistence, or bistability \cite{Hofbauer1998}:

\begin{itemize}
\item if $\lambda_1^2>0$ and $\lambda_2^1>0$, the interior equilibrium $z_{eq}^*$ is stable and both species coexist;
\item if $\lambda_1^2$ and $\lambda_2^1$ have opposite signs, competitive exclusion occurs, with one boundary equilibrium stable and the other unstable;
\item if $\lambda_1^2<0$ and $\lambda_2^1<0$, the system is bistable: both boundary equilibria $z_0^*, z_1^*$ are locally stable, while the interior equilibrium $z_{eq}^*$ is unstable and separates the basins of attraction.
\end{itemize}

\subsection{Spatial replicator dynamics with heterogeneity}

The formulation above assumes a well-mixed population and therefore neglects
spatial heterogeneity. In many ecological and biological systems, however, this
assumption is unrealistic: individuals interact locally, dispersal processes
shape encounter rates, and environmental variation modifies local fitness
landscapes.

Motivated by this, we extend the classical two-species replicator equation by
allowing the interaction coefficients $\lambda_1^2(x)$ and $\lambda_2^1(x)$ to
depend explicitly on spatial location $x\in\Omega\subset\mathbb{R}^n$. We also
incorporate diffusion, representing random dispersal or local movement, and
advection, representing directed transport induced, for example, by flow or
environmental drift. The resulting reaction-advection-diffusion equation for
the frequency of species~1 is
\begin{equation}
\begin{cases}
\displaystyle
\frac{\partial z}{\partial t}
=
z(1-z)\bigl(\lambda_1^2(x)-(\lambda_1^2(x)+\lambda_2^1(x))z\bigr)
+
\vec{\nu}(x)\cdot\nabla z
+
\nabla\cdot\!\bigl(D(x)\nabla z\bigr),
\\[0.2cm]
\displaystyle
\frac{\partial z}{\partial \vec n}=0,
\quad x\in\partial\Omega,
\\[0.1cm]
z(x,0)=z_0(x)\in L^2(\Omega).
\end{cases}
\label{eq:repinspace}
\end{equation}
Here $\Omega\subset\mathbb{R}^n$ is the spatial domain, $D(x)$ is the diffusion coefficient, and $\vec{\nu}(x)$ is the advection field. The no-flux boundary condition means that individuals do not leave the domain through the boundary.

We assume the following conditions:

\begin{enumerate}[label=\textbf{(H\arabic*)}]
\item \label{H:domain}
$\Omega\subset\mathbb{R}^n$ is bounded with $C^{1,1}$ boundary.

\item \label{H:d}
$D(\cdot)\in C^{0,1}(\overline\Omega)$;

\item \label{H:nu}
$\vec{\nu}\in W^{1,\infty}(\Omega;\mathbb{R}^n)$;

\item \label{H:lambda}
$\lambda_1^2,\lambda_2^1\in W^{1,\infty}(\Omega)$; 

\item \label{H:lambda2}
\( \lambda_1^2 \geq 0\) a.e. in $\Omega$ and $\lambda_1^2(x)>0$ on a nonempty open subset of $\Omega$.

\end{enumerate}

\subsubsection{Links with the homogeneous model}

When the interaction coefficients are constant in space, the spatial model recovers the classical replicator behavior, except in the bistable regime.

\begin{proposition}[Spatially-homogeneous interaction coefficients]
\label{prop:constant_lambda_main}
Assume $\lambda_1^2(x)\equiv\lambda_1^2$ and $\lambda_2^1(x)\equiv\lambda_2^1$ are constants. Let $z$ be a solution for \eqref{eq:repinspace} with $0\le z_0\le1$ and $z_0\not\equiv0,1$. Then:

\begin{enumerate}[label=(\roman*)]
\item if $\lambda_1^2>0$ and $\lambda_2^1<0$, species~1 excludes species~2, that is, $z(\cdot,t)\to1$;
\item if $\lambda_1^2<0$ and $\lambda_2^1>0$, species~2 excludes species~1, that is, $z(\cdot,t)\to0$;
\item if $\lambda_1^2>0$ and $\lambda_2^1>0$, the system admits the stable spatially homogeneous coexistence equilibrium
\[
z_{eq}^*=\frac{\lambda_1^2}{\lambda_1^2+\lambda_2^1}\in(0,1).
\]

\item If $\lambda_1^2<0$ and $\lambda_2^1<0$, the system is multistable and the outcome depends on the initial condition; in this regime spatial structure may influence the equilibrium.
\end{enumerate}

\end{proposition}

\begin{proof}
We give the argument in case (i); the other cases follow by analogous comparison arguments.

\smallskip
\noindent\textbf{(i): $\lambda_1^2>0$ and $\lambda_2^1<0$.}
Let $z$ solve \eqref{eq:repinspace}. By the strong maximum principle and the Neumann boundary condition,
if $z_0\not\equiv 0$ then $z(x,t)>0$ for all $x\in\overline\Omega$ and all $t>0$.
Fix $t_0>0$. Choose $\varepsilon>0$ such that
\[
0<\varepsilon \le \min_{x\in\overline\Omega} z(x,t_0).
\]
Define the spatially constant function $\underline z(x,t)=\zeta(t)$ for $t\ge t_0$ as the solution of the ODE
\begin{equation}\label{eq:sub_ode_casei}
\frac{d\zeta}{dt}
=
\zeta(1-\zeta)\bigl(\lambda_1^2-(\lambda_1^2+\lambda_2^1)\zeta\bigr),
\qquad
\zeta(t_0)=\varepsilon.
\end{equation}
Since $\zeta$ is constant in space, $\nabla \zeta\equiv 0$ and $\nabla\cdot(D(x)\nabla \zeta)\equiv 0$,
and therefore $\zeta$ solves \eqref{eq:repinspace} with the same reaction term. In particular, $\zeta$ is a solution
of the PDE \eqref{eq:repinspace} and hence a subsolution.

The ODE \eqref{eq:sub_ode_casei} is exactly the two-species replicator equation with constant coefficients.
When $\lambda_1^2>0$ and $\lambda_2^1<0$, its unique stable equilibrium in $(0,1]$ is $\zeta=1$,
and any solution with $\zeta(t_0)>0$ satisfies $\zeta(t)\rightarrow 1$ as $t\to\infty$.

By construction, $\zeta(\cdot,t_0)\le z(\cdot,t_0)$ in $\Omega$.
The parabolic comparison principle then implies
\[
\zeta(t)\le z(x,t)\le 1
\qquad\text{for all } x\in\overline\Omega,\ t\ge t_0,
\]
and so $z(\cdot,t)\to 1$ as $t\to\infty$.

\smallskip
\noindent\textbf{(ii)}
Case (ii) follows by applying the same argument to $w=1-z$, which satisfies an equation of the same form and converges to $w\to 1$,
i.e.\ $z\to 0$. 

\noindent\textbf{(iii)}
Case (iii) is obtained by constructing spatially constant sub- and supersolutions from the scalar ODE, with initial values chosen below and above both $z(\cdot,t_0)$ and the coexistence equilibrium $z_{eq}^*$. These two ODE solutions converge to $z_{eq}^*$ from below and above, respectively, and the comparison principle then gives $z(\cdot,t)\to z_{eq}^*$.

For the supersolution, for example, choose $\delta>0$ such that
\[
z(x,t_0)\le 1-\delta
\qquad\text{for all }x\in\overline\Omega .
\]
Let $\overline z(x,t)=\xi(t)$, where $\xi$ solves
\[
\frac{d\xi}{dt}
=
\xi(1-\xi)\bigl(\lambda_1^2-(\lambda_1^2+\lambda_2^1)\xi\bigr),
\qquad
\xi(t_0)=1-\delta .
\]
Since $\overline z$ is spatially constant, the diffusion and advection terms vanish. Moreover, $\overline z$ solves  \eqref{eq:repinspace}, and is a supersolution.

Thus, since $\lambda_1^2>0$ and $\lambda_2^1>0$, we have that the supersolution will converge to
\[
z_{eq}^*=\frac{\lambda_1^2}{\lambda_1^2+\lambda_2^1}\in(0,1)
\]
based on the non-spatial replicator. A similar argument can be applied to show that the subsolution also converges to $z_{eq}^*$.
\end{proof}

Figure~\ref{fig:changeswithspace} summarizes the four regimes described in Proposition~\ref{prop:constant_lambda_main}.

\begin{remark}[Multistable case $\lambda_1^2<0$ and $\lambda_2^1<0$.]
In this regime, the well-mixed dynamics have two locally stable equilibria at $0$ and $1$ with an unstable interior threshold. It is easy to understand that there are at least two stable solutions, similar to the non-spatial case, depending on whether the initial condition is above or below the unstable threshold. However, the addition of space adds the potential for more coexistence states to arise. For example, in the absence of diffusion and advection, the equation reduces to a collection of uncoupled ODEs, one at each spatial point. This creates a continuum of stationary spatial profiles. Which profile is selected depends on the initial data $z_0(x)$, and its position  relative to the unstable threshold.
\end{remark}

\subsubsection{Global invasion fitness}

In the spatial model, the coefficients $\lambda_i^j(x)$ represent local invasion fitnesses.
A global measure of invasion is obtained from the principal eigenvalue of the linearized
operator around the boundary equilibria.

Linearizing \eqref{eq:repinspace} around $z=0$ yields

\begin{equation}
\begin{cases}
\nabla \cdot\!\bigl(D(x)\nabla \psi\bigr)
+\vec{\nu}(x)\cdot\nabla \psi
+\lambda_1^2(x)\psi
= \sigma \psi,
& x\in\Omega,
\\
\displaystyle
\frac{\partial \psi}{\partial \vec n}=0,
& x\in\partial\Omega.
\end{cases}
\label{eq:eigenproblem}
\end{equation}
Under assumptions \textbf{\ref{H:domain}-\ref{H:lambda2}}, classical results for elliptic
operators with Neumann boundary conditions ensure that the problem
\eqref{eq:eigenproblem} admits a principal eigenvalue associated with a strictly
positive eigenfunction (see, e.g., \cite{Berestycki1994,Cantrell2003}).
\begin{definition}[Spatial invasion fitness]
Let $\rho_1^2$ denote the principal eigenvalue of \eqref{eq:eigenproblem}. When $\vec{\nu}=0$, this principal eigenvalue admits the characterization
\begin{equation}
\rho_1^2
=
\sup_{\varphi\in H^1(\Omega),\ \varphi\not\equiv 0}
\frac{
\displaystyle
\int_\Omega \lambda_1^2(x)\varphi^2\,dx
-
\int_\Omega D(x)|\nabla\varphi|^2\,dx
}{
\displaystyle
\int_\Omega \varphi^2\,dx
}.
\end{equation}
When advection is present, the operator is generally non-self-adjoint. In that case, a variational characterization is impractical, but the principal eigenvalue remains well defined and retains this interpretation.
\end{definition}
Similarly, defining $w=1-z$ and linearizing around $w=0$ yields the corresponding principal eigenvalue $\rho_2^1$.

Species $i$ can invade species $j$ when rare if and only if $\rho_i^j>0$.

The quantities $\rho_1^2$ and $\rho_2^1$ therefore act as spatial invasion
fitnesses for the heterogeneous replicator system. This suggests a direct
analogy with the classical model, where the qualitative dynamics
are determined by the signs of $(\lambda_1^2,\lambda_2^1)$. One might therefore
expect the spatial dynamics to be governed entirely by the sign structure of
$(\rho_1^2,\rho_2^1)$.

Under this heuristic, the same four regimes would persist after replacing the
local invasion fitnesses by their spatial counterparts: competitive exclusion
when the signs are opposite, coexistence when both are positive, and bistability
when both are negative.

However, this intuition is incomplete. Spatial heterogeneity can fundamentally
alter the global behavior and may allow coexistence even in parameter regimes
where the homogeneous replicator predicts exclusion. In the next section we
explore how the dynamics change across the $(\rho_1^2,\rho_2^1)$ plane, before explaining this phenomenon through a backward bifurcation mechanism in
Section~\ref{sec:backwards}.

\section{Exploring the $(\rho_1^2,\rho_2^1)$ plane}

We now examine how the qualitative dynamics of the spatial system depend on
the sign structure of $(\rho_1^2,\rho_2^1)$ and compare the resulting regimes
with those of the well-mixed replicator equation.

\subsection{Quadrant I: $\rho_1^2>0,\ \rho_2^1>0$}

We begin with the case where both species can invade when rare. As in the
non-spatial replicator equation, this regime leads to persistence of both
species.

\begin{lemma}[Persistence from a positive invasion rate]
\label{lem:persistence}
Let $z(x,t)$ be a solution of \eqref{eq:repinspace} with
$0 \le z_0 \not\equiv 0$.
\begin{enumerate}
\item If $\rho_1^2>0$, then species~1 persists uniformly: there exist
$\delta>0$ and $T>0$ such that
\[
z(x,t)\ge \delta
\quad \text{for all } x\in\overline{\Omega},\ t\ge T.
\]
\item If $\rho_2^1>0$, then species~2 persists uniformly, i.e.
$1-z(x,t)\ge\delta$ for all $x$ and sufficiently large $t$.
\end{enumerate}
\end{lemma}

\begin{proof}
    This result can be shown by building a subsolution for $z$, defined as $\underline{z}$. We begin by rewriting the reaction term:

\begin{align}
    h(x,z) &:= z(1-z)(\lambda_1^2(x)-(\lambda_1^2(x)+ \lambda_2^1(x))z) = z(1-z)(\lambda_1^2(x)(1-z)-z  \lambda_2^1(x)) \\ &= z[\lambda_1^2(x)-2z\lambda_1^2(x) + z^2 \lambda_1^2(x) - z \lambda_2^1(x)].
\end{align}

Let $\psi_1>0$ be a positive eigenfunction for \eqref{eq:eigenproblem}, which exists because of the Krein-Rutman theorem. Then, for $\epsilon>0$ sufficiently small,
\begin{align}
    &\nabla \cdot D(x) \nabla (\epsilon \psi_1)  + \vec{\nu} \cdot (x) \nabla (\epsilon \psi_1) + h(x, \epsilon \psi_1) \\
    &=\epsilon[\nabla \cdot D(x) \nabla \psi_1  + \vec{\nu}(x)\cdot\nabla  \psi_1 + \psi_1 \lambda_1^2(x)] + (\epsilon \psi_1)^2(-2\lambda_1^2 + \epsilon \psi_1 \lambda_1^2(x) + \psi_1  \lambda_2^1(x)) \\ 
    &= \epsilon \rho_1^2 \psi_1 +(\epsilon \psi_1)^2(-2\lambda_1^2 +  \epsilon\psi_1 \lambda_1^2(x) + \psi_1 \lambda_2^1(x)) \\
    &= \epsilon \psi_1[\rho_1^2 + \epsilon \psi_1 (-2\lambda_1^2 +  \epsilon\psi_1 \lambda_1^2(x) + \psi_1 \lambda_2^1(x))] > 0,
\end{align}
given that we choose $\epsilon>0$ small enough such that $$\epsilon \Big|\max_{x \in \Omega} \psi_1\big(-2 \lambda_1^2(x) + \epsilon \psi_1 \lambda_1^2(x)+\psi_1 \lambda_2^1(x)\big)\Big|<\rho_2^1,$$
which is possible because of \ref{H:lambda}. 

It follows that $\epsilon \psi_1$ is a subsolution for the elliptic problem
\begin{equation}
    \begin{cases}
    z(1-z)(\lambda_1^2(x)-(\lambda_1^2(x)+ \lambda_2^1(x))z) + \vec{\nu}(x) \cdot \nabla z + \nabla \cdot D(x) \nabla z =0, \quad x\in \Omega\\
     \frac{\partial z}{\partial \vec{n}} = 0, \quad x\in \partial \Omega \\
    \end{cases}
\end{equation}
corresponding to \eqref{eq:repinspace}. Finally, if $\underline{z}(x,t)$ is a solution to \eqref{eq:repinspace} with $\underline{z}(x,0) = \epsilon\psi_1$, then $\frac{\partial \underline{z}}{\partial t}|_{t=0}>0$, and $\underline{z}(x,t)$ is increasing in $t$. If $z(x,t)$ is a solution to \eqref{eq:repinspace}, which is initially nonnegative and is positive on an open subset of $\Omega$, then by the strong maximum principle, $z(x,t)>0$ on $\overline\Omega$ for $t>0$. Choosing any $t_0>0$, we may take $\epsilon>0$ small enough so that $\epsilon \psi_1< z(x,t_0)$ on $\overline\Omega$. Then $\underline{z}(x,t-t_0)< z(x,t)$ for $t=t_0$, and by the maximum principle, for $t>t_0$, so $z(x,t)$ is bounded below by $\underline{z}(x,t-t_0) > \min_x \epsilon \psi_1 = \delta$.
\end{proof}

This resembles the coexistence regime of the classical replicator
dynamics, although with one key difference: steady states are not uniquely
determined by the coefficients $\lambda_i^j(x)$. In particular,
there can be regions where the local interaction is bistable, i.e., where
both $\lambda_1^2(x)$ and $\lambda_2^1(x)$ are negative, while keeping $\rho_1^2, \rho_2^1>0$. In those regions,
the local reaction term has two stable boundary equilibria, separated by an
unstable interior threshold. Diffusion then couples these locally bistable regions with the rest of the domain, so the limiting steady state can depend on the spatial structure of the coefficients and on the initial condition. As a result, different stable attractors are possible.
Figure~\ref{fig:1stquadrant} illustrates one example, where there are at least two of these steady states.

\begin{figure}[htbp]
    \centering
    \includegraphics[width=.99\textwidth]{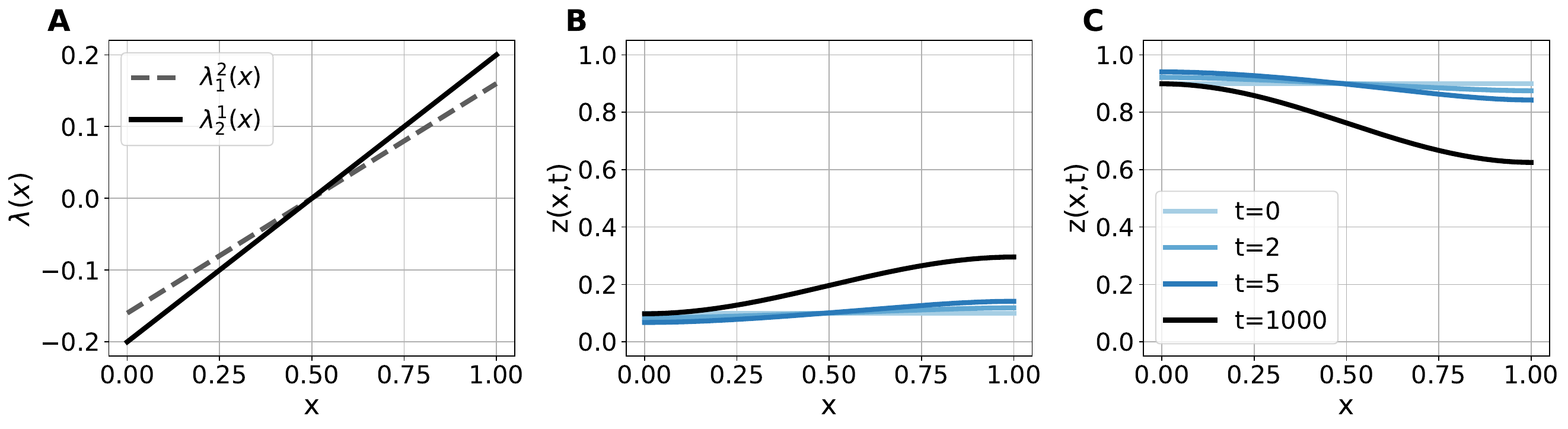}
    \caption{Numerical simulations in the first quadrant ($\rho_1^2>0$, $\rho_2^1>0$), for $D=0.01$, $\vec{\nu}=0$, showing coexistence steady states that depend on the initial condition.   
\textbf{(A)} Local interaction coefficients: 
$\lambda_1^2(x)=-0.16 + 0.32x$ and $\lambda_2^1(x)=-0.2 + 0.4x$ 
($\rho_1^2\approx0.06$, $\rho_2^1\approx0.08$). 
\textbf{(B,C)} Evolution of $z(x,t)$ for the functions in (A), starting from $z(x,0)=0.1$ (B) and $z(x,0)=0.9$ (C). }
    \label{fig:1stquadrant}
\end{figure}

\subsection{Quadrants II and IV: opposite invasion signs}

We now consider the regions where the invasion criteria have opposite signs:
\[
\text{Quadrant II: } \rho_1^2>0,\ \rho_2^1<0,
\qquad
\text{Quadrant IV: } \rho_1^2<0,\ \rho_2^1>0.
\]

In the well-mixed replicator model, these sign patterns lead to competitive
exclusion. The spatial system retains part of this structure but also admits
qualitatively new behavior.

We first identify a sufficient condition that guarantees the classical competitive-exclusion outcome in these quadrants.

\begin{proposition}[Exclusion under sign-definite interactions]
\label{prop:sign_definite_exclusion}
Assume that $\lambda_1^2(x)$ and $\lambda_2^1(x)$ do not change sign in space.
\begin{itemize}
\item If $\lambda_1^2(x)>0$ and $\lambda_2^1(x)<0$ for all $x\in\overline\Omega$,
then species~1 excludes species~2.
\item If $\lambda_1^2(x)<0$ and $\lambda_2^1(x)>0$ for all $x\in\overline\Omega$,
then species~2 excludes species~1.
\end{itemize}
\end{proposition}

\begin{proof}
The proof follows from spatially constant sub- and super-solution arguments
analogous to those used in Proposition~\ref{prop:constant_lambda_main}.
\end{proof}

\begin{remark}[Spatial heterogeneity can reverse exclusion]
\label{rem:spatial_reversal}
Proposition~\ref{prop:sign_definite_exclusion} shows that the classical replicator outcome is recovered whenever the interaction coefficients keep a fixed sign. Combined with Lemma~\ref{lem:persistence}, this gives an important constraint: positive spatial coexistence in Quadrants~II and IV cannot arise from sign-definite interaction coefficients alone. It necessarily relies on at least one of the coefficients $\lambda_1^2(x)$ or $\lambda_2^1(x)$ changing sign in space.
Once such spatial sign changes are present, the qualitative dynamics need not follow the classical well-mixed prediction. In particular, the spatial model may admit positive coexistence states even in parameter regimes where one of the spatial invasion rates is negative. This phenomenon is characterized rigorously in Theorem~\ref{thm:local_direction}, where spatial heterogeneity gives rise to a backward bifurcation and allows stable positive states to persist beyond the classical invasion threshold.
\end{remark}

Intuitively, this happens because spatial variation can create localized regions where different species have a competitive advantage, allowing a coexistence state to emerge even when one spatial invasion rate is negative.

Figure~\ref{fig:2nd4thquadrant}
illustrates dynamics and steady states in this regime for fixed $\lambda_i^j(x)$, where depending on the initial condition, the system can either converge to a coexistence state or to $z=0$.

\begin{figure}[htbp]
    \centering
    \includegraphics[width=.99\textwidth]{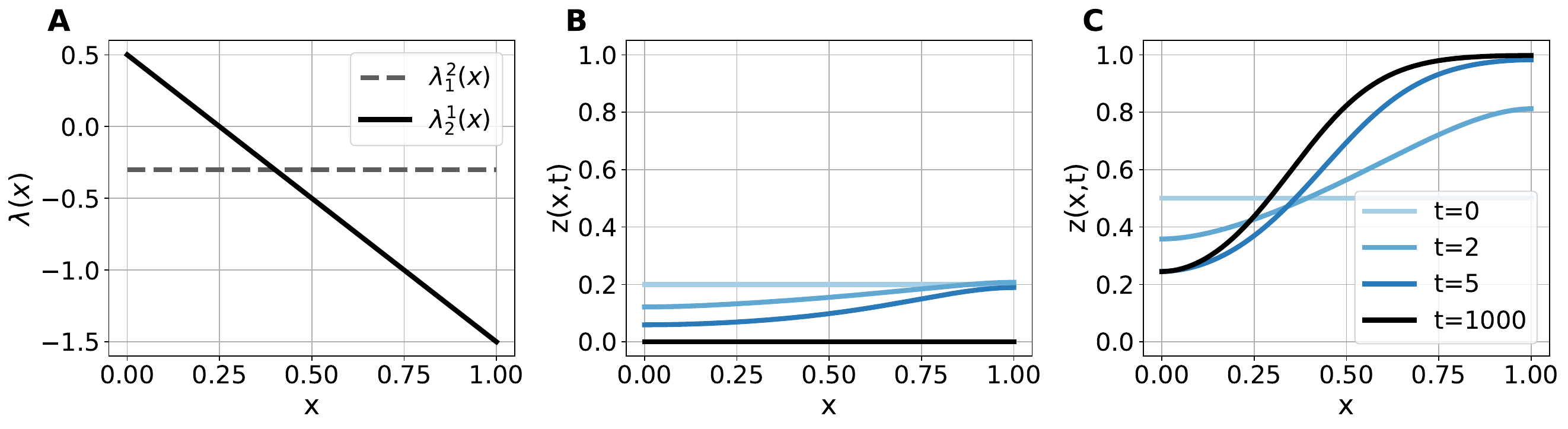}
    \caption{Numerical simulations in the second quadrant ($\rho_1^2<0$, $\rho_2^1>0$), for $D=0.01$, $\vec{\nu}=0$, showing coexistence steady states that depend on the initial condition.   
\textbf{(A)} Local interaction coefficients: 
$\lambda_1^2(x)=-0.3$ and $\lambda_2^1(x)=0.5 - 2x$ 
($\rho_1^2\approx-0.30$, $\rho_2^1\approx0.15$). 
\textbf{(B,C)} Evolution of $z(x,t)$ for the functions in (A), starting from $z(x,0)=0.2$ (B) and $z(x,0)=0.5$ (C). }
    \label{fig:2nd4thquadrant}
\end{figure}

\subsection{Quadrant III: $\rho_1^2<0,\ \rho_2^1<0$}

When both invasion rates are negative, neither species can invade the other
when rare. As in the classical replicator dynamics, both boundary equilibria
are locally stable.


In the spatial setting, however, the dynamics can be richer. Spatial
heterogeneity allows for the construction of intermediate steady states
between the two boundary equilibria. This is not surprising: in the absence
of diffusion and advection, the system reduces to uncoupled replicator
equations at each spatial point, leading to a continuum of reachable steady states
determined by the initial condition.

The more interesting point is that some of these coexistence states can remain stable under the full spatial dynamics. In particular, diffusion can help support stable nontrivial spatial steady states.

Figure~\ref{fig:3rdquadrant} illustrates such a scenario, where the system
admits at least three distinct stable steady states depending on the initial
condition. The boundary equilibria $z=0$ and $z=1$, shown in panels (A) and
(C), are expected from the negative invasion rates. In contrast, the stable
coexistence state shown in panel (B) is a new spatial effect and has
no analogue in the well-mixed system.

\begin{figure}[htbp]
    \centering
    \includegraphics[width=.99\textwidth]{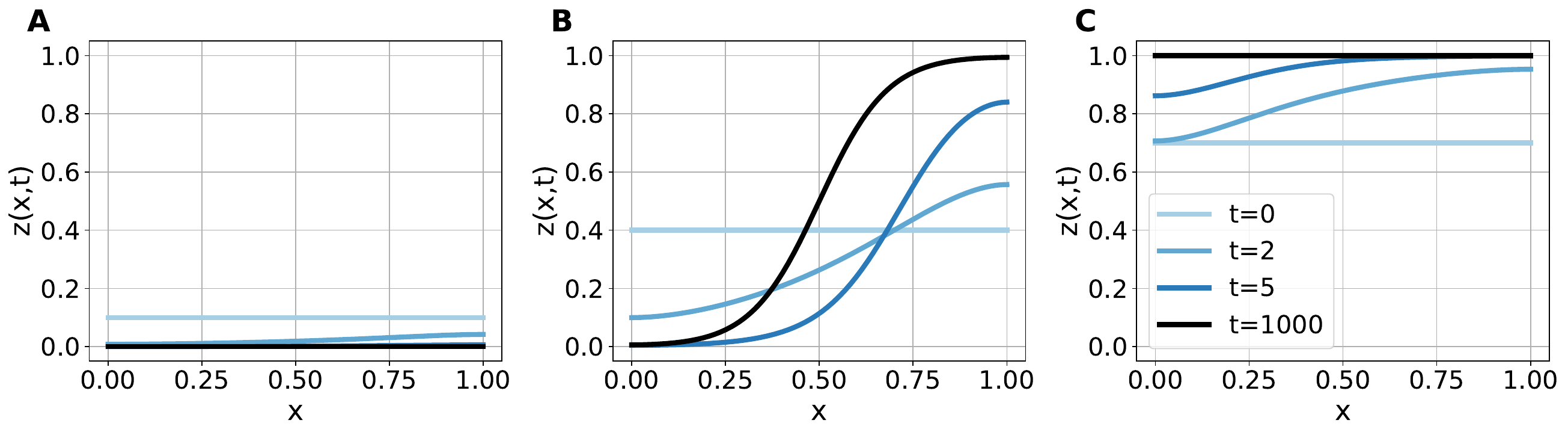}
    \caption{Numerical simulations in the third quadrant ($\rho_1^2<0$, $\rho_2^1<0$), for $D=0.01$, $\vec{\nu}=0$, showing three stable steady states: $z=0$, a coexistence state, and $z=1$.  
\textbf{(A–C)} Evolution of $z(x,t)$ for the local interaction coefficients 
$\lambda_1^2(x)=-1.5+x$ and $\lambda_2^1(x)=-0.5-x$ 
($\rho_1^2\approx-0.72$, $\rho_2^1\approx-0.72$), starting from $z(x,0)=0.1$ (A), $z(x,0)=0.4$ (B), and $z(x,0)=0.7$ (C). }
    \label{fig:3rdquadrant}
\end{figure}

\medskip

In summary, Figure~\ref{fig:changeswithspace},
Quadrants~I and III largely preserve the qualitative structure of
the non-spatial replicator dynamics (persistence and bistability,
respectively), whereas the most significant changes occur in Quadrants~II
and IV, where spatial heterogeneity can enable coexistence despite opposite
invasion signs. The next section develops the bifurcation framework that
explains this phenomenon.

\begin{figure}[H]
    \centering
    \includegraphics[width=.99 \textwidth]{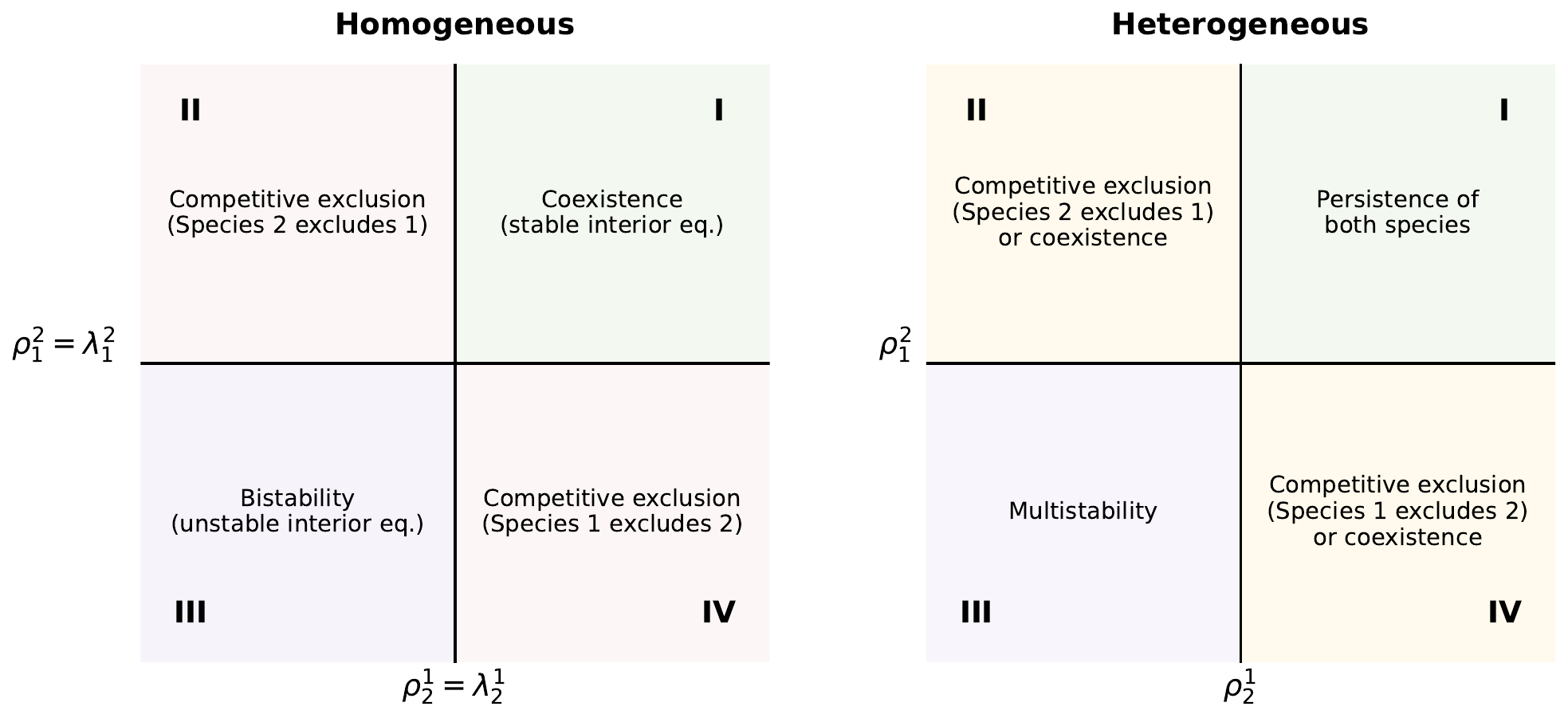}
    \caption{Two-species replicator regimes. The quantities $\lambda_i^j$ and $\rho_i^j$ denote the pairwise invasion fitness of species $i$ in an environment set by species $j$ in the homogeneous and heterogeneous replicator models, respectively.}
    \label{fig:changeswithspace}
\end{figure}

\section{Coexistence through a backward bifurcation}
\label{sec:backwards}

In this section we develop the bifurcation framework that explains the
coexistence regimes observed in Quadrants~II and IV of the
$(\rho_1^2,\rho_2^1)$ plane. The key idea is to introduce a one-parameter
family obtained by scaling $\lambda_1^2$ by a scalar factor $a\in\mathbb R$ and to
track the emergence of nontrivial steady states from the equilibrium
$z\equiv 0$.

\subsection{Bifurcation formulation}
\label{subsec:functional_setup}

Let $a \in \mathbb{R}$. We consider the spatial replicator equation
\begin{equation}
\label{eq:repinspace_with_a}
\begin{cases}
\displaystyle
\frac{\partial z}{\partial t}
=
z(1-z)\Big(a\lambda_1^2(x)-\big(a\lambda_1^2(x)+\lambda_2^1(x)\big)z\Big)
+\vec{\nu}(x)\cdot\nabla z
+\nabla\cdot\!\big(D(x)\nabla z\big),
\\[0.2cm]
\displaystyle
\frac{\partial z}{\partial \vec n}=0,
\quad x\in\partial\Omega,
\\[0.1cm]
z(x,0)=z_0(x).
\end{cases}
\end{equation}

Steady states satisfy
\begin{align}
\label{eq:steady_eq}
0
&=
z(1-z)\Big(a\lambda_1^2(x)-\big(a\lambda_1^2(x)+\lambda_2^1(x)\big)z\Big)
+\vec{\nu}(x)\cdot\nabla z
+\nabla\cdot\!\big(D(x)\nabla z\big) \\
&= g(x,z,a) + \mathcal L_a z,
\end{align}
where $g$ is the quadratic function
\begin{equation}
    g(x,z,a) := 
-\bigl(2a\lambda_1^2(x)+\lambda_2^1(x)\bigr)z
+
\bigl(a\lambda_1^2(x)+\lambda_2^1(x)\bigr)z^2.
\end{equation}
 and the operator $\mathcal L_a$ is defined as
\begin{equation}
\label{eq:La_def}
\mathcal L_a \psi
:=
\nabla\cdot\!\big(D(x)\nabla \psi\big)
+\vec{\nu}(x)\cdot\nabla \psi
+a\lambda_1^2(x)\psi,
\qquad
\partial_{\vec n}\psi=0.
\end{equation}

\paragraph{Preliminaries for bifurcation analysis.}
To apply standard bifurcation theory, it is convenient to rewrite the equilibrium \eqref{eq:steady_eq} in a fixed-point form.

Let $\mathcal T$ be the linear elliptic operator defined by
\[
\mathcal T\varphi
:=
-\nabla\cdot\!\big(D(x)\nabla \varphi\big)
-\vec{\nu}(x)\cdot\nabla \varphi,
\qquad
\partial_{\vec n}\varphi=0.
\]
with domain $D(\mathcal T)=\{\varphi\in W^{2,p}(\Omega):\partial_{\vec n}\varphi=0\}$. Under hypothesis \textbf{\ref{H:domain}-\ref{H:nu}}, $\mathcal T$ is a uniformly elliptic operator with bounded coefficients. Its spectrum in $L^p(\Omega)$ is discrete and consists of isolated eigenvalues with finite multiplicity.

Since $\mathcal T$ is not invertible under Neumann boundary conditions, we introduce the shifted operator
\[
A := \mathcal T - c\,\lambda_1^2(x) I,
\]
where $c>0$ is chosen so that $A$ is invertible in $L^p(\Omega)$. This choice is possible because the spectrum of the weighted eigenvalue problem
is discrete. By elliptic regularity, the inverse operator satisfies

\[A^{-1}:L^p(\Omega)\to W^{2,p}(\Omega) \quad \text{bounded}.
\]
For $p>n$, the Sobolev embedding $$W^{2,p}(\Omega)\hookrightarrow C^1(\overline\Omega)$$ is continuous and compact. 
Consequently,
\[A^{-1}:L^p(\Omega)\to C^1(\overline\Omega) \]
is a compact operator. 

Subtracting $c\,\lambda_1^2(x)z$ from both sides of \eqref{eq:steady_eq}
yields
\begin{equation}\label{eq:shifted_eq_C1}
A z
= (a-c)\lambda_1^2(x)z
+ \bigl(g(x,z,a)-a\lambda_1^2(x)\bigr)z.
\end{equation}

Setting $\mu:=a-c$ and applying $A^{-1}$ to both sides we obtain
\begin{equation}\label{eq:352}
z = \mu \underbrace{A^{-1}\lambda_1^2(x) z \,}_{=:L(z)}
\;+\;
\underbrace{A^{-1}\Big(\bigl(g(x, z,a)-a\lambda_1^2(x)\bigr)z\Big)}_{=:H(\mu,z)}.
\end{equation}

Since $\lambda_1^2\in W^{1,\infty} \subset L^\infty(\Omega)$ and
$C^1(\overline\Omega)\hookrightarrow L^\infty(\Omega)$,
multiplication by $\lambda_1^2$ defines a bounded operator from $
C^1(\overline\Omega)$ to $L^p(\Omega).$
Together with elliptic regularity and the compact embedding
$W^{2,p}(\Omega)\hookrightarrow C^1(\overline\Omega)$, we conclude that
\[
L : C^1(\overline\Omega)\longrightarrow C^1(\overline\Omega)
\]
is a compact linear operator.

As for the nonlinear term $H(a,z)$, since $g$ is $C^1$ in $z$ and satisfies
$g(\cdot,0,a) = a \lambda_1^2(\cdot)$, we can write
\[
g(\cdot,z,a) - a \lambda_1^2(\cdot) = \mathcal O(z) \quad \text{as } \|z\|_{C^1(\overline\Omega)} \to 0,
\]
uniformly for $a$ in bounded intervals.  
Multiplying by $z$ and applying the bounded operator $A^{-1}:L^p(\Omega) \to C^1(\overline\Omega)$, we obtain
\[
\|H(a,z)\|_{C^1} \le C \|z\|_{C^1}^2,
\]
and so,
\[
\frac{\|H(a,z)\|_{C^1}}{\|z\|_{C^1}} \longrightarrow 0 \quad \text{as } \|z\|_{C^1} \to 0,
\]
uniformly for $a$ in any bounded interval.

Therefore it is natural to work in the Banach space
\[
Y := C^1_N(\overline\Omega)
:= \{z\in C^1(\overline\Omega)\;:\;\partial_{\vec n}z=0\ \text{on }\partial\Omega\}.
\]

\subsection{Global bifurcation structure}
\label{subsec:global_bifurcation}

We begin by characterizing the spectral properties of the linearized
operator \eqref{eq:La_def}, which determine the bifurcation from
the trivial branch.

\begin{proposition}[Principal eigenvalue]
\label{prop:principal_eigen_La}
Under assumptions \textbf{\ref{H:domain}-\ref{H:lambda2}}, the eigenvalue problem
\begin{equation}
\label{eq:eigen_La}
\begin{cases}
\mathcal L_a \psi = \sigma \psi, & x\in\Omega,\\
\partial_{\vec n}\psi=0, & x\in\partial\Omega,
\end{cases}
\end{equation}
admits a principal eigenvalue, denoted
$\sigma_0(\mathcal L_a)$, which is algebraically simple and associated
with a strictly positive eigenfunction. Any other eigenfunction changes sign
in $\Omega$.
\end{proposition}

\begin{proof}
This follows from the Krein-Rutman theorem applied to the elliptic eigenvalue problem; see, for example, Theorem~2.12 \cite{Cantrell2003}.
\end{proof}

The principal eigenvalue $\sigma_0(\mathcal L_a)$ depends continuously on
the parameter $a$. We define $a_0$ as the point where
\[
\sigma_0(\mathcal L_{a_0})=0.
\]

\begin{definition}[Characteristic values of $L$]
\label{rem:characteristic_value}
The number $\mu$ is a characteristic value of
$L$ if there is a nonzero $\phi \in Y$ such that
$\phi=\mu L\phi$, or, in other words, if $1/\mu$ is a nonzero eigenvalue of the compact operator $L$.
\end{definition}

We can now identify the bifurcation point for nontrivial steady states.

\begin{proposition}[Bifurcation from the trivial branch]
\label{prop:bifurcation_point}
Assume hypotheses \textbf{\ref{H:domain}-\ref{H:lambda2}} hold. Let
$a_0$ satisfy $\sigma_0(\mathcal L_{a_0})=0$. Then $(a_0,0)$ is a
bifurcation point of nontrivial solutions of the steady-state problem
\eqref{eq:steady_eq} in the sense of Rabinowitz's global bifurcation
theorem (Theorem \ref{thm:rabinowitzSI}).
\end{proposition}

\begin{proof}
Application of Proposition 3.9 of \cite{Cantrell2003}.
\end{proof}

This means that there is a connected continuum $\mathcal C$ of nontrivial solutions $(a,z)\in \mathbb R \times Y$ of \eqref{eq:steady_eq} that bifurcates from the trivial equilibrium at $(a_0,0)$. We now study the structure of this bifurcation branch, showing that it is unbounded in the $a$-direction and does not meet the trivial branch at any other value. For this, we will make use of the following lemma:

\begin{lemma}
    \label{lemma:tours}
    Assume that \(\rho_2^1>0\). Then there exists \(a_{\mathrm{crit}}<0\)
    such that, for every \(a<a_{\mathrm{crit}}\), any solution of
    \eqref{eq:repinspace_with_a} with biologically admissible initial data
    \[
    0\le z_0(x)\le 1, 
    \qquad z_0\not\equiv 1,
    \]
    converges to \(z=0\).
\end{lemma}

\begin{proof}
Using the persistence result from Lemma~\ref{lem:persistence}, we know that if $\rho_2^1>0$, then species 2 persists uniformly in space, and there exists $\delta>0$ such that the density of species 1, $z<1-\delta$ for all $x\in \Omega$ and $T$ big enough. Now, we will show that for any given $\lambda_1^2, \lambda_2^1$, we can build a super-solution going to $0$, based on our knowledge of the non-spatial replicator dynamics by choosing a negative enough $a$. 

Set
\[
\overline z(x, 0)=1-\delta .
\]
Since \(\overline z\) is chosen to be spatially constant, the diffusion and
advection terms vanish. We now bound the reaction term from above. Write
\[
z(1-z)\Big(a\lambda_1^2(x)
-\big(a\lambda_1^2(x)+\lambda_2^1(x)\big)z\Big)
=
z(1-z)\Big(a\lambda_1^2(x)(1-z)-\lambda_2^1(x)z\Big).
\]
Let
\[
\underline{\lambda_1^2}
:=
\min_{x\in\Omega}\lambda_1^2(x),
\qquad
\underline{\lambda_2^1}
:=
\min_{x\in\Omega}\lambda_2^1(x).
\]
For \(a<0\),
\[
a\lambda_1^2(x)
\le
a\underline{\lambda_1^2}, \quad -\lambda_2^1(x)z
\le
-\underline{\lambda_2^1}z.
\]
Hence, for \(0\le z\le 1\),
\[
z(1-z)\Big(a\lambda_1^2(x)(1-z)-\lambda_2^1(x)z\Big)
\le
z(1-z)\Big(a\underline{\lambda_1^2}(1-z)
-\underline{\lambda_2^1}z\Big).
\]

We now define \(\overline z(t)\) as the solution of the scalar ODE
\[
\frac{d\overline z}{dt}
=
\overline z(1-\overline z)
\Big(a\underline{\lambda_1^2}(1-\overline z)
-\underline{\lambda_2^1}\overline z\Big),
\qquad
\overline z(0)=1-\delta .
\]
By the inequality above, \(\overline z(t)\) is a supersolution of
\eqref{eq:repinspace_with_a}. Since
\[
z(x,T)\le 1-\delta=\overline z(0),
\]
the comparison principle gives
\[
z(x,t)\le \overline z(t)
\qquad \text{for all }x\in\Omega,\ t\ge T .
\]

It remains to choose \(a<0\) so that \(\overline z(t)\to0\). If
\(\underline{\lambda_2^1}\ge0\), then
\[
a\underline{\lambda_1^2}(1-\overline z)
-\underline{\lambda_2^1}\overline z<0
\qquad \text{for all }0<\overline z<1,
\]
and therefore \(\overline z(t)\to0\) for every \(a<0\).

If \(\underline{\lambda_2^1}<0\), the scalar equation is bistable. Its
unstable threshold is
\[
\mu_a
=
\frac{a\underline{\lambda_1^2}}
{a\underline{\lambda_1^2}+\underline{\lambda_2^1}}.
\]
Since \(a<0\), \(\underline{\lambda_1^2}>0\), and
\(\underline{\lambda_2^1}<0\), we have
\[
\mu_a\to1
\qquad \text{as }a\to-\infty .
\]
Thus, we can choose \(a_{\mathrm{crit}}<0\) such that
\[
\mu_a>1-\delta
\qquad \text{for all }a<a_{\mathrm{crit}}.
\]
For such values of \(a\), the initial value
\(\overline z(0)=1-\delta\) lies below the unstable threshold \(\mu_a\), and
therefore the scalar ODE satisfies
\[
\overline z(t)\to0 .
\]
Since 
\(
z(x,t)\le \overline z(t),
\)
\[
z(\cdot,t)\to0
\qquad \text{as }t\to\infty .
\]
\end{proof}

\begin{remark}
Lemma~\ref{lemma:tours} depends on the sign structure imposed in \textbf{(H5)}. In particular, the proof uses that taking \(a<0\) makes \(a\lambda_1^2(x)\) act in the same way throughout the domain. This is a limitation of the comparison argument used here, and the case of sign-changing \(\lambda_1^2\) is not covered by the present proof.

This Lemma should not be interpreted as excluding the possibility of other multistable regimes with \(\rho_1^2<0\) and \(\rho_2^1>0\). Rather, if such regimes occur, they are not explained by the mechanism isolated here. 
\end{remark}

\begin{theorem}[Global structure of the positive bifurcation branch] \label{thm:thm2}
Under the assumptions stated above, let $\mathcal C$ denote the connected component 
of nontrivial solutions $(a,z)\in \mathbb R\times Y$ of \eqref{eq:steady_eq}
that bifurcates from the trivial equilibrium at $(a_0,0)$, then:

\begin{enumerate}[label=(\roman*)]
\item \emph{(Uniform pointwise bound)}  
If $(a,z)\in\mathcal C$ with $z\ge0$, then 
\[
0 \le z(x) \le 1 \qquad \text{for all } x\in \overline\Omega.
\]

\item \emph{(Exclusion of hypothesis ii from Theorem \ref{thm:rabinowitzSI})}  
$\mathcal C$ does not meet the trivial branch 
$\{(a,0): a\in\mathbb R\}$ at any parameter value $a \neq a_0$.


\item \emph{(Unboundedness in the \(a\)-direction)}
The connected component $\mathcal C$ of nontrivial nonnegative steady states bifurcating from $(a_0,0)$ is unbounded in $\mathbb R\times Y$. Moreover, this unboundedness occurs in the a-direction, in the sense that
\[
\sup\{a:(a,z)\in\mathcal C\}=+\infty .
\]

\end{enumerate}
\end{theorem}

\begin{proof}

\textbf{(i)}

The constant functions $\underline z\equiv 0$ and
$\overline z\equiv 1$ are respectively sub- and super-solutions
of \eqref{eq:steady_eq}. The comparison principle yields
$0\le z\le 1$.

\textbf{(ii)}

We show that the principal eigenvalue of \(\mathcal L_a\) can vanish for at most one value of \(a\). 

Fix \(a_1\) and suppose \(\sigma(\mathcal L_{a_1})=0\).
Let \(\phi_{a_1}>0\) be the principal eigenfunction of \(\mathcal L_{a_1}\), and
let \(\psi_{a_1}^*>0\) be the principal eigenfunction of the adjoint operator
\(\mathcal L_{a_1}^*\),
\begin{equation}
\label{eq:La_adjoing}
\mathcal L_{a_1}^* \psi^*
:=
\nabla\cdot \big(D(x)\nabla \psi^*\big)
-\nabla\cdot \big(\vec{\nu}(x)\psi^*\big)
+a_1\lambda_1^2(x)\psi^*.
\end{equation}

Assume that there exists \(a_2 \neq a_1\) such that
\(\sigma(\mathcal L_{a_2})=0\), and let \(\phi_{a_2}>0\) be the corresponding
principal eigenfunction of \(\mathcal L_{a_2}\).

For any two numbers \(a_1,a_2\) we have 
\[
\mathcal L_{a_2}
= \mathcal L_{a_1} + (a_2-a_1)\,\lambda_1^2(\cdot).
\]

Now we multiply by the eigenfunction \(\phi_{a_2}\). Since
\(\mathcal L_{a_2}\phi_{a_2}=0\) (principal eigenvalue \(0\)), we have
\[
0 = \mathcal L_{a_2}\phi_{a_2}
= \mathcal L_{a_1}\phi_{a_2} + (a_2-a_1)\,\lambda_1^2\,\phi_{a_2}.
\]
Taking the \(L^2(\Omega)\) inner product of both sides with \(\psi_{a_1}^*\)
and using linearity of the inner product 
\[
0
= \langle \mathcal L_{a_2}\phi_{a_2}, \psi_{a_1}^* \rangle
= \langle \mathcal L_{a_1}\phi_{a_2}, \psi_{a_1}^* \rangle
+ (a_2-a_1)\,\langle \lambda_1^2\,\phi_{a_2}, \psi_{a_1}^* \rangle.
\]

Now, we use the adjoint relation:
\[
\langle \mathcal L_{a_1}\phi_{a_2}, \psi_{a_1}^* \rangle
= \langle \phi_{a_2}, \mathcal L_{a_1}^*\psi_{a_1}^* \rangle.
\]
Since \(\mathcal L_{a_1}^*\psi_{a_1}^*=0\), the first term vanishes and we obtain
\[
0 = (a_2-a_1)\,\langle \lambda_1^2\,\phi_{a_2}, \psi_{a_1}^* \rangle.
\]

By the Krein-Rutman theorem, the eigenfunctions \(\phi_{a_2}\) and
\(\psi_{a_1}^*\) are strictly positive in \(\Omega\). Under
the hypothesis \ref{H:lambda2} \(\lambda_1^2>0\) a.e. in $\Omega$ and $\lambda_1^2>0$ on a nonempty open subset of $\Omega$, so the
integral
\(\langle \lambda_1^2\,\phi_{a_2}, \psi_{a_1}^* \rangle\) is strictly positive.
Therefore  \((a_2-a_1)=0\), and $a_2=a_1$.

\textbf{(iii)}

Since alternative (2) in Theorem~\ref{thm:rabinowitzSI} has been excluded
by \textbf{(ii)}, the remaining possibility is that the component
\(\mathcal C\) is unbounded in \(\mathbb R\times Y\).

We now show that this unboundedness cannot occur through the \(Y\)-norm while
\(a\) remains bounded. Let \((a,z)\in\mathcal C\), with \(a\) in a bounded
interval \([a_-,a_+]\). By \textbf{(i)}, we have the uniform pointwise bound
\[
0\le z(x)\le1
\qquad \text{for all }x\in\overline\Omega .
\]
Define the nonlinear term
\[
f(x,z,a)
:=
z(1-z)
\Big(
a\lambda_1^2(x)
-
(a\lambda_1^2(x)+\lambda_2^1(x))z
\Big).
\]
Since \(z\) is bounded between \(0\) and \(1\), and \(a\) is in a bounded
interval, \(f\) is uniformly bounded in \(L^\infty(\Omega)\). Therefore
\[
\|f\|_{L^\infty(\Omega)}\le C_0
\]
for some constant \(C_0\).

The steady-state equation can be written as
\[
-\nabla\cdot(D(x)\nabla z)-\vec\nu(x)\cdot\nabla z
=
f(x,z,a)
\qquad \text{in }\Omega,
\qquad
\partial_{\vec n}z=0
\qquad \text{on }\partial\Omega .
\]
Using the elliptic estimate for a strong elliptic operator,  Theorem 9.11 \cite{Gilbarg2001}, for \(p>n\),
\[
\|z\|_{W^{2,p}(\Omega)} \le C \bigl( \|Lz\|_{L^p(\Omega)} + \|z\|_{L^p(\Omega)} \bigr).
\]
Since $Lz = f$, it follows that
\begin{equation}
\|z\|_{W^{2,p}(\Omega)} \le C \bigl( \|f\|_{L^p(\Omega)} + \|z\|_{L^p(\Omega)} \bigr).
\end{equation}
Since \(0\le z\le1\), and since \(f\) is uniformly bounded, this gives
\[
\|z\|_{W^{2,p}(\Omega)}\le C .
\]
Finally, the Sobolev embedding
\[
W^{2,p}(\Omega)\hookrightarrow C^1(\overline\Omega)
\]
implies
\[
\|z\|_{Y}\le C ,
\]
with \(C\) independent of \(a\) in bounded intervals. Thus, along
\(\mathcal C\), the \(Y\)-norm of \(z\) cannot blow up while \(a\) remains
bounded.

It remains to exclude the possibility that the component extends to
\(a\to-\infty\). By Lemma~\ref{lemma:tours}, there exists
\(a_{\mathrm{crit}}<0\) such that, for \(a<a_{\mathrm{crit}}\), all solutions
of \eqref{eq:repinspace_with_a} with biologically admissible initial data
converge to \(z=0\). In particular, there can be no nontrivial biologically
admissible steady state for \(a<a_{\mathrm{crit}}\).

Therefore \(\mathcal C\) is unbounded, its \(Y\)-norm cannot become unbounded
for bounded \(a\), and it cannot extend to \(a\to-\infty\). The only remaining
possibility is that
\[
\sup\{a:(a,z)\in\mathcal C\}=+\infty .
\]

\end{proof}

\subsection{Local bifurcation and direction of branching}
\label{subsec:local_bifurcation}

We now describe the local structure of the bifurcating branch
near the bifurcation point $(a_0,0)$. 

\begin{proposition}[Local parametrisation of the bifurcating branch]
\label{prop:local_parametrisation}
Let $\sigma_0(\mathcal L_{a_0})=0$ be the principal eigenvalue of
\eqref{eq:eigen_La} and let $\psi_0>0$ be the associated eigenfunction.
Let $Z \subset W^{2,p}(\Omega)$ be any closed complement of
$\operatorname{span}\{\psi_0\}$. Then there exists a neighbourhood of
$(a_0,0)$ in $\mathbb R\times W^{2,p}(\Omega)$ in which the set of
nontrivial solutions of \eqref{eq:steady_eq} consists of a $C^1$
curve
\[
(a(s),z(s)) = \bigl(a(s),\, s\psi_0 + s\phi(s)\bigr),
\]
where $s$ varies in a neighbourhood of $0$, $a_0=a(0)$,
$\phi(0)=0$, and $\phi(s)\in Z$.
\end{proposition}

\begin{proof}
Application of Proposition 3.10 in \cite{Cantrell2003}.
\end{proof}

Thus, near $(a_0,0)$, the solution set consists of the trivial branch $(a,0)$
and a smooth curve of nontrivial solutions bifurcating from it.

For the slope computation, we take
\[
Z=\left\{\phi\in W^{2,p}(\Omega):
\int_\Omega \psi_0\,\phi\,dx = 0\right\},
\]
the $L^2$–orthogonal complement of $\operatorname{span}\{\psi_0\}$.
Then, for $|s|<\delta$ small,
\begin{equation}
\label{eq:localsol}
(a(s),z(s)) = \bigl(a(s),\, s\psi_0 + s\phi(s)\bigr).
\end{equation}

The direction of the bifurcation is determined by the local slope
$a_s(0)$ at $(a_0,0)$. If the bifurcating branch bends toward
smaller values of $a$, the bifurcation is backward.
In this case, positive steady states exist for some $a<a_0$,
yielding coexistence in parameter regimes where the homogeneous
replicator dynamics predict competitive exclusion.
The next result provides an explicit formula for the slope $a_s(0)$
in terms of the principal eigenfunctions of $\mathcal L_{a_0}$
and its adjoint $\mathcal L_{a_0}^*$.

\begin{theorem}[Direction of bifurcation]
\label{thm:local_direction}
Let $\sigma_0(\mathcal L_{a_0})=0$ be simple and let
$\psi_0>0$ be the associated eigenfunction.
Let $(a(s),z(s))$ be the local branch described in
Proposition~\ref{prop:local_parametrisation}.
Then

\begin{equation}
    a_s(0) = \frac{\int_\Omega \lambda_2^1 \psi_0^2\psi_0^* dx}{\int_\Omega \lambda_1^2  \psi_0\psi_0^* dx}.
\end{equation}

In particular, if $a_s(0)<0$, the bifurcation is backward and
positive, coexistence steady states exist for $a<a_0$.
\end{theorem}


\begin{proof}
Substituting the local parametrization \eqref{eq:localsol} into
\eqref{eq:repinspace_with_a} and differentiating with respect to \(s\),
we obtain
\begin{equation}
\label{eq:first_s_derivative_Dx}
\begin{cases}
\nabla \cdot\!\Big(D(x)\nabla z_s\Big)
+ \vec{\nu}\cdot \nabla z_s
+ (z g)_s = 0,\\[0.2cm]
\displaystyle \frac{\partial z_s}{\partial \vec{n}}=0.
\end{cases}
\end{equation}

Differentiating again with respect to \(s\) yields
\begin{equation}
\label{eq:second_s_derivative_Dx}
\begin{split}
\nabla \cdot\!\Big(D(x)\nabla z_{ss}\Big)
+ \vec{\nu}\cdot \nabla z_{ss}
+ (z g)_{ss} = 0 .
\end{split}
\end{equation}

To compute the last term, write
\[
zg
=
z(1-z)\left(a\lambda_1^2(1-z)-\lambda_2^1 z\right).
\]
Then, at \(s=0\), using
\[
z(0)=0,\qquad z_s(0)=\psi_0,\qquad z_{ss}(0)=2\phi_s(0),
\qquad a_0=a_0,
\]
we obtain
\begin{equation}
\label{eq:zgss_at_zero_Dx}
(zg)_{ss}\big|_{s=0}
=
a_0\lambda_1^2 z_{ss}(0)
+
2a_s(0)\lambda_1^2\psi_0
-
\left(4a_0\lambda_1^2+2\lambda_2^1\right)\psi_0^2 .
\end{equation}
Therefore, setting \(s=0\) in \eqref{eq:second_s_derivative_Dx}, we obtain
\begin{equation}
\label{eq:second_derivative_at_zero_Dx}
\nabla \cdot\!\Big(D(x)\nabla z_{ss}(0)\Big)
+ \vec{\nu}\cdot \nabla z_{ss}(0)
+ a_0\lambda_1^2 z_{ss}(0)
+ 2a_s(0)\lambda_1^2\psi_0
-
\left(4a_0\lambda_1^2+2\lambda_2^1\right)\psi_0^2
=0 .
\end{equation}

Now, let \(\psi_0^*>0\) be a principal eigenfunction of the adjoint operator
\(\mathcal L_{a_0}^*\) corresponding to the eigenvalue \(0\):
\begin{equation}
\begin{cases}
\nabla \cdot\!\big(D(x) \nabla \psi_0^* \big)
- \nabla \cdot\!\big(\vec{\nu}(x)\psi_0^*\big)
+ a_0 \lambda_1^2(x)\, \psi_0^* = 0,
& x \in \Omega,\\[0.2cm]
\displaystyle \frac{\partial \psi_0^*}{\partial \vec{n}} = 0,
& x \in \partial \Omega.
\end{cases}
\label{eq:eigenproblemadjoint_Dx}
\end{equation}

Multiplying \eqref{eq:second_derivative_at_zero_Dx} by \(\psi_0^*\) and
integrating over \(\Omega\), and applying integration by parts, we obtain
\begin{align}
&\int_\Omega z_{ss}(0)
\Big[
\nabla \cdot\!\big(D(x)\nabla \psi_0^*\big)
- \nabla \cdot\!\big(\vec{\nu}\psi_0^*\big)
+ a_0\lambda_1^2 \psi_0^*
\Big]dx \notag\\
&\qquad
+ 2 a_s(0)\int_\Omega \lambda_1^2\,\psi_0\psi_0^*\,dx
-
\int_\Omega
\left(4a_0\lambda_1^2+2\lambda_2^1\right)
\psi_0^2\psi_0^*\,dx
=0.
\end{align}

By \eqref{eq:eigenproblemadjoint_Dx}, the bracketed term multiplying
\(z_{ss}(0)\) vanishes. Hence
\begin{equation}
\label{eq:close_Dx}
2 a_s(0)\int_\Omega \lambda_1^2\,\psi_0\psi_0^*\,dx
=
\int_\Omega
\left(4a_0\lambda_1^2+2\lambda_2^1\right)
\psi_0^2\psi_0^*\,dx .
\end{equation}
As we are considering Neumann boundary conditions, $a_0=0$, and we obtain
\begin{equation}
\label{eq:as0_formula_Dx}
a_s(0)
=
\frac{
\displaystyle
\int_\Omega
\lambda_2^1
\psi_0^2\psi_0^*\,dx
}{
\displaystyle
\int_\Omega \lambda_1^2\,\psi_0\psi_0^*\,dx
}.
\end{equation}
\end{proof}

We note one simplification and one extension of this expression:

\begin{itemize}
\item If $\vec{\nu}(x)=0$, then the eigenvalue problem is self-adjoint, so
$\psi_0=\psi_0^*$, and
\begin{equation}
\label{eq:backconditionsimpler}
a_s(0)
=
\frac{
\int_\Omega \lambda_2^1 \psi_0^3\,dx
}{\int_\Omega \lambda_1^2\,\psi_0^2\,dx}.
\end{equation}
\item If $D=D(x,z)$, similar computations yield:
\begin{equation}
    a_s(0) = \frac{\int_\Omega D_z(x,0) \psi_0 \nabla \psi_0 \cdot \nabla \psi_0^* + \int_\Omega \lambda_2^1 \psi_0^2\psi_0^* dx}{\int_\Omega \lambda_1^2  \psi_0\psi_0^* dx},
\end{equation}
lthough more regularity in the coefficients is required. The computations are provided in the SI Text \ref{SI:bifurcationdz}.
\end{itemize}

If $a_s(0)>0$, the positive branch $\mathcal C^+$
emerges toward larger values of $a$, and the local dynamics
near $(a_0,0)$ are consistent with the classical
competitive exclusion picture.

Conversely, if $a_s(0)<0$, the branch bends toward smaller parameter values, and the bifurcation is backward. Since the denominator in the expression for \(a_s(0)\) is positive, this corresponds to a negative weighted average of \(\lambda_2^1\), with weights given by the principal eigenfunction and its adjoint. In this case, positive steady states exist for some $a<a_0$. Combined with the global result of Theorem~\ref{thm:thm2}, this implies the existence of coexistence states in parts of Quadrants II and IV of the $(\rho_1^2,\rho_2^1)$ plane, in contrast with the homogeneous replicator dynamics described in Proposition~\ref{prop:constant_lambda_main}.

\section{Examples and insights on coexistence and invasibility}

We now illustrate several mechanistic features of the spatial replicator system that have direct ecological interpretation. We examine how  diffusion, advection, and spatial heterogeneity shape both coexistence and resistance to invasion.

\subsection{Backward bifurcation and invasion resistance}

Figure~\ref{fig:backwardexample}A shows a representative backward bifurcation diagram obtained for
$\lambda_1^2(x)=1$ and $\lambda_2^1(x)=0.5-2x$, with diffusion coefficient
$D=0.01$, and no advection. When $a<0$, the system is in Quadrant II of the
$(\rho_1^2,\rho_2^1)$ plane. We observe the emergence of a stable coexistence
branch that extends into the region where the well-mixed dynamics predict
competitive exclusion.

\begin{figure}[htbp]
    \centering
    \includegraphics[width=.99\textwidth]{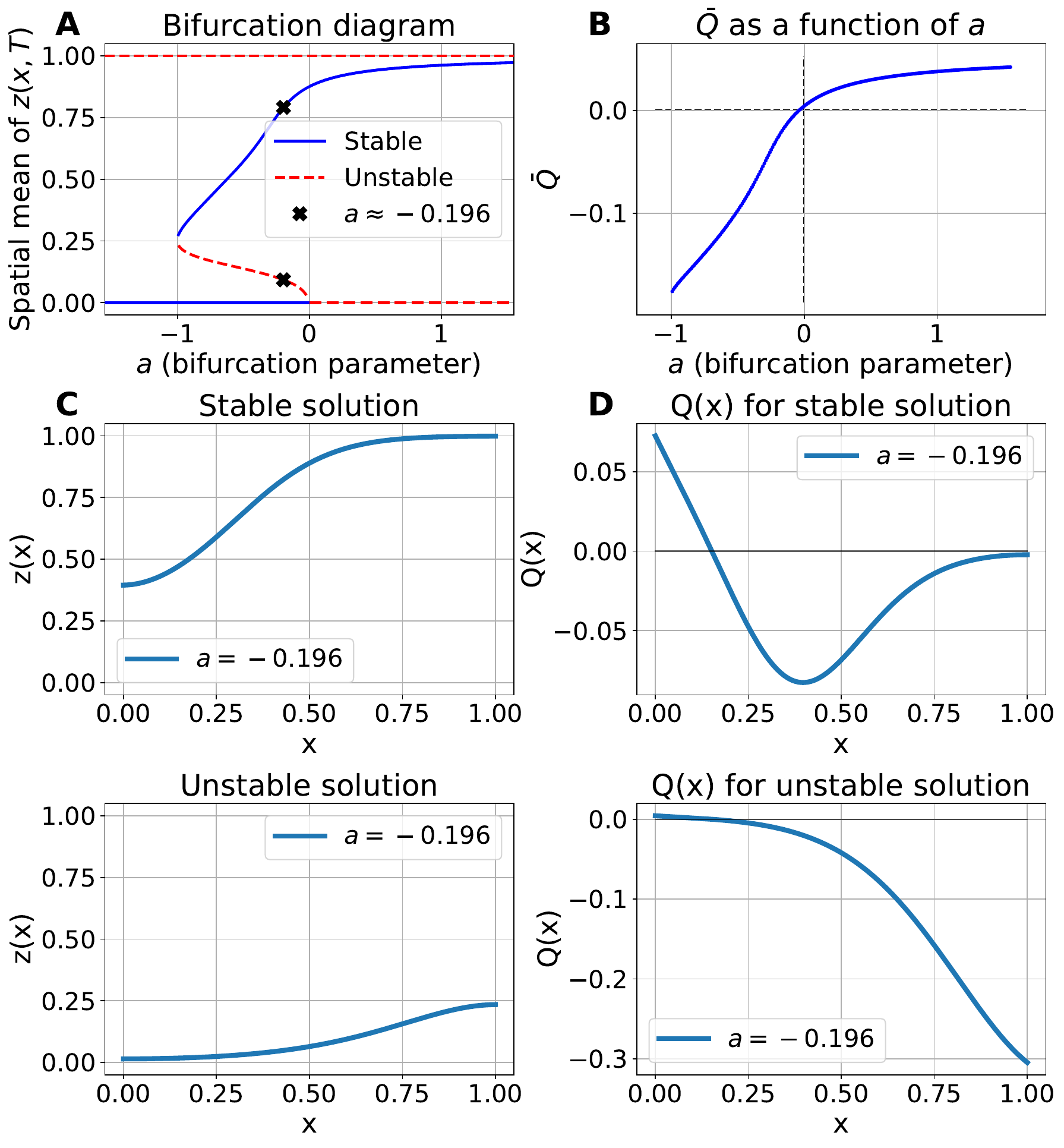}
    \caption{\textbf{(A)} Backward Bifurcation diagram for $\lambda_1^2(x)=1$, $\lambda_2^1(x)=0.5-2x$, and $D=0.01$. Red indicates the unstable steady states, while blue represents the stable ones. We plot the mean abundance of $z$ as a function of the bifurcation parameter $a$. These results were obtained by simulating the system until $T=1000$. \textbf{(B)} Mean system invasion resistance to invasion, $\bar{Q}$, as a function of $a$. \textbf{(C)} Plot of the stable and unstable of the solution for $a\approx -0.2$. \textbf{(D)} Plot of the corresponding invasion resistance for the stable and unstable solution.}
    \label{fig:backwardexample}
\end{figure}

\subsubsection*{Invasion resistance to a third species}

Although species~1 and species~2 coexist along the stable branch, we find an interesting phenomenon when considering invasion by a third species. Real ecosystems are rarely closed, and even a stable resident state can be vulnerable to the introduction of an additional species. In the spatial model, this is even more important, and we find the existence of stable coexistence states whose mean resistance to invasion can be negative.

To make this explicit, we consider a third species, initially rare, trying to
invade the resident state
\[
    \mathbf z = (z,1-z,0).
\]
The initial growth rate of this third species depends on its direct interactions
with each resident species, given by $\lambda_3^1(x)$ and $\lambda_3^2(x)$, but
also on the invasion resistance generated by the resident pair. In the
two-species case, this resistance is
\begin{equation}
    Q_a(\mathbf z,x)
    =
    \bigl(a\lambda_1^2(x)+\lambda_2^1(x)\bigr)z(1-z).
\end{equation}
Therefore, the local growth rate of the rare third species is
\begin{equation}
    r_3(\mathbf z,x)
    =
    z\lambda_3^1(x)
    +
    (1-z)\lambda_3^2(x)
    -
    Q_a(\mathbf z,x).
\end{equation}

The important point is that $Q_a$ is generated by the resident two-species
system. It is not a property of the invader itself. Positive values of $Q_a$
reduce the growth rate of a rare third species, while negative values of $Q_a$
make invasion easier. Thus, the same coexistence state can be internally stable
with respect to species~1 and species~2, but still be susceptible to invasion by
a third species.

To quantify this effect along the coexistence branch, we compute the mean
invasion resistance
\begin{equation}
    \overline Q_a
    =
    \frac{1}{|\Omega|}
    \int_\Omega
    \bigl(a\lambda_1^2(x)+\lambda_2^1(x)\bigr)z(1-z)\,dx,
\end{equation}
as a function of the bifurcation parameter $a$
(Figure~\ref{fig:backwardexample}B).

Despite the persistence of dynamically stable coexistence over a wide range of
$a$, we find that $\overline Q_a$ can become negative for some negative values
of $a$. In other words, backward bifurcation does not only create a coexistence
state between species~1 and species~2. It can also create resident communities
that are stable internally, but that facilitate invasion by an additional
species.

This contrasts with the homogeneous model, where stability of the resident
equilibrium and resistance to invasion are more tightly linked. In the spatial
model, these two notions can be separated.

On the other hand, for non-negative values of $a$, we can show analytically that
$\overline Q_a \geq 0$ in the absence of advection. Numerical exploration further suggests that this remains true when advection is present.

Let
\begin{align}
Q_{a=0}
:= \int_\Omega \lambda_2^1(x)\, z^*(1-z^*)\,dx.
\end{align}

Using the steady-state equation and integrating by parts yields
\begin{align}
\int_\Omega \lambda_2^1(x)\, z^*(1-z^*)\,dx
= \int_\Omega D \frac{|\nabla z^*|^2}{(z^*)^2}\,dx
+ D \int_{\partial\Omega} \frac{\nabla z^*}{z^*}.
\end{align}

Under no-flux boundary conditions, the boundary term vanishes, and therefore
\begin{align}
Q_{a=0} \geq 0.
\end{align}

\subsection{Effect of diffusion}
As diffusion increases, the system becomes progressively more well-mixed. Using the same conditions as in Figure~\ref{fig:backwardexample}, we observe that the backward bifurcation branch shrinks as diffusion grows. For sufficiently large diffusion, the system transitions to a multistable regime in which both $z=0$ and $z=1$ are stable, and the coexistence branch disappears, Figure~\ref{fig:effectsofdiffusion}.

This transition occurs because for $D=0.04$, the principal eigenvalue $\rho_2^1$ becomes negative (while $\rho_1^2$ remains negative for negative values of $a$), so we are in the multistable case, with atractors on the boundary.

\begin{figure}[htbp]
    \centering
    \includegraphics[width=0.7\linewidth]{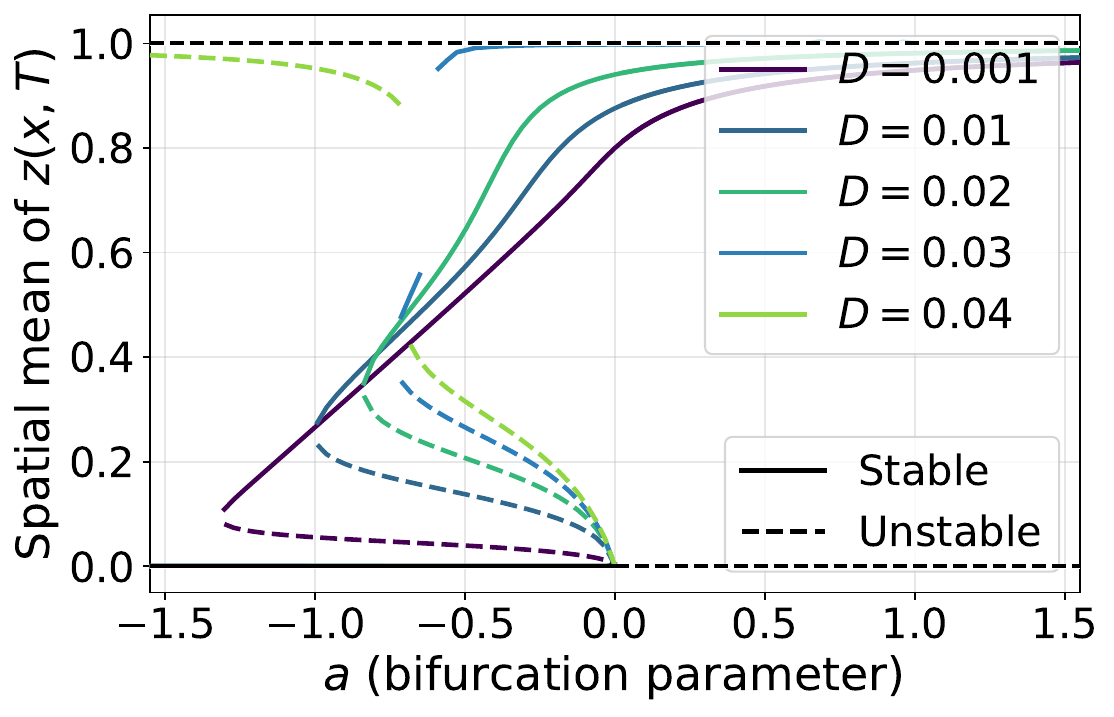}
    \caption{Effects of diffusion. Bifurcation branches for D=0.001, D=0.01, D=0.02, D=0.03, D=0.04. Other parameters as in Figure \ref{fig:backwardexample}.}
    \label{fig:effectsofdiffusion}
\end{figure}

\subsection{Effect of advection}

Advection affects the system differently from diffusion. In particular, its impact on the bifurcation structure is non-monotonic: depending on its magnitude, advection can either induce or eliminate a backward bifurcation. 

For example, using a different set of $\lambda_i^j$, we find that a backward bifurcation emerges at intermediate levels of advection, even though it is absent for both lower and higher values, Figure~\ref{fig:effectsofadvection}. We can observe three different scenarios: for $\nu=0$, $z=0$ and $z=1$ are stable steady states, and there is no stable coexistence branch. For $\nu=-0.05$, a backward bifurcation emerges, and a stable coexistence branch appears for negative values of $a$. Finally, for $\nu=-0.1$, the backward bifurcation disappears again, and is replaced by a forward bifurcation.

\begin{figure}[htbp]
    \centering
    \includegraphics[width=0.7\linewidth]{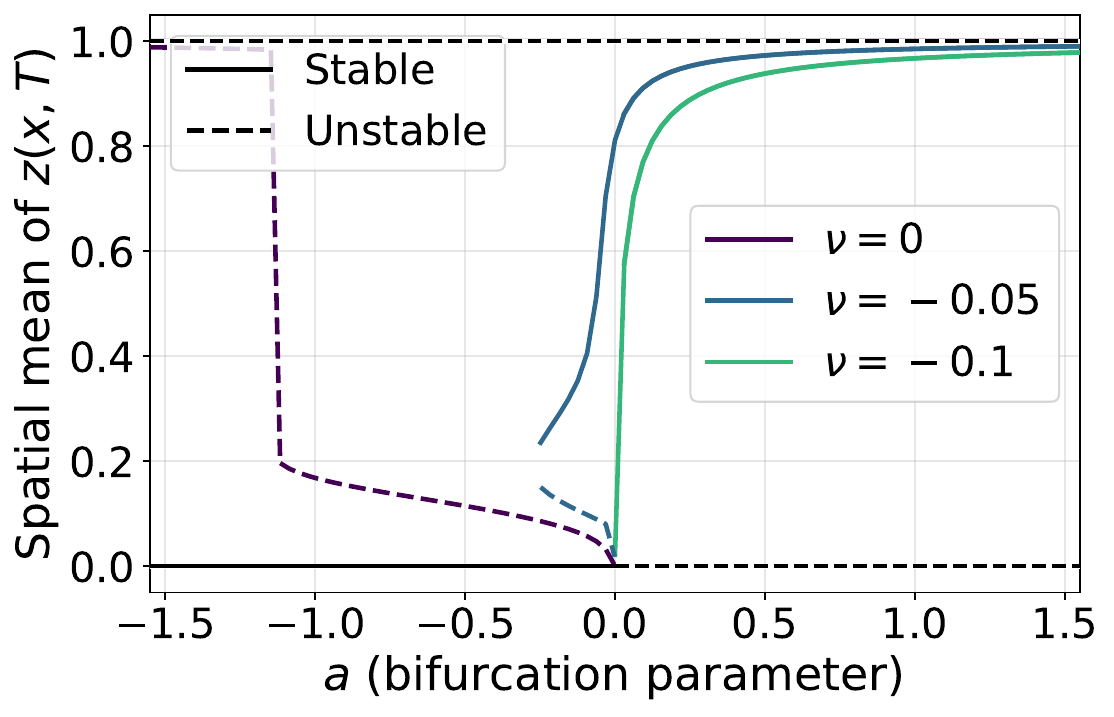}
    \caption{Effects of advection: Backward Bifurcation diagram for $\lambda_1^2(x)=1$, $\lambda_2^1(x)=0.1 - 2 x^2$, and $D=0.01$ for different advection coefficients: $\vec{\nu}=0$, $\vec{\nu}=-0.05$, $\vec{\nu}=-0.1$.}
    \label{fig:effectsofadvection}
\end{figure}

\subsection{Example of regions in the \((\rho_1^2,\rho_2^1)\) plane}

So far, we have seen when backward bifurcations arise and the types of coexistence regimes they generate. However, it is still not entirely clear how these regimes are reflected in the \((\rho_1^2,\rho_2^1)\) plane.

In an attempt to clarify this, we explore the \((\rho_1^2,\rho_2^1)\) plane across a family of examples obtained by varying two parameters (\(a,b\)). We take
\[
\lambda_1^2(x)=a,
\qquad
\lambda_2^1(x)=b-2x.
\]
For \(b=0.5\), this recovers the setting of Figure~\ref{fig:backwardexample}. For each value of \(b\), we compute the smallest value \(a_{\mathrm{crit}}(b)\) for which coexistence of both species is possible. Since \(\lambda_1^2(x)=a\) is constant in space, we have \(\rho_1^2=a\), and therefore the curve
\[
\big(a_{\mathrm{crit}}(b),\rho_2^1(b)\big)
\]
can be represented directly in the \((\rho_1^2,\rho_2^1)\) plane.

Figure~\ref{fig:planefinal2} shows one representative example for these parameters, with \(D=0.01\) and \(\vec{\nu}=0\). In the second quadrant, we can see an added region of coexistence, region III', extending the multistability of Quadrant III into Quadrant II. For the other quadrants, \(a_{\mathrm{crit}} = 0\) in Quadrants I or IV and \(a_{\mathrm{crit}} = -\infty\) in Quadrant III.


\begin{figure}[H]
    \centering
    \includegraphics[width=0.6\linewidth]{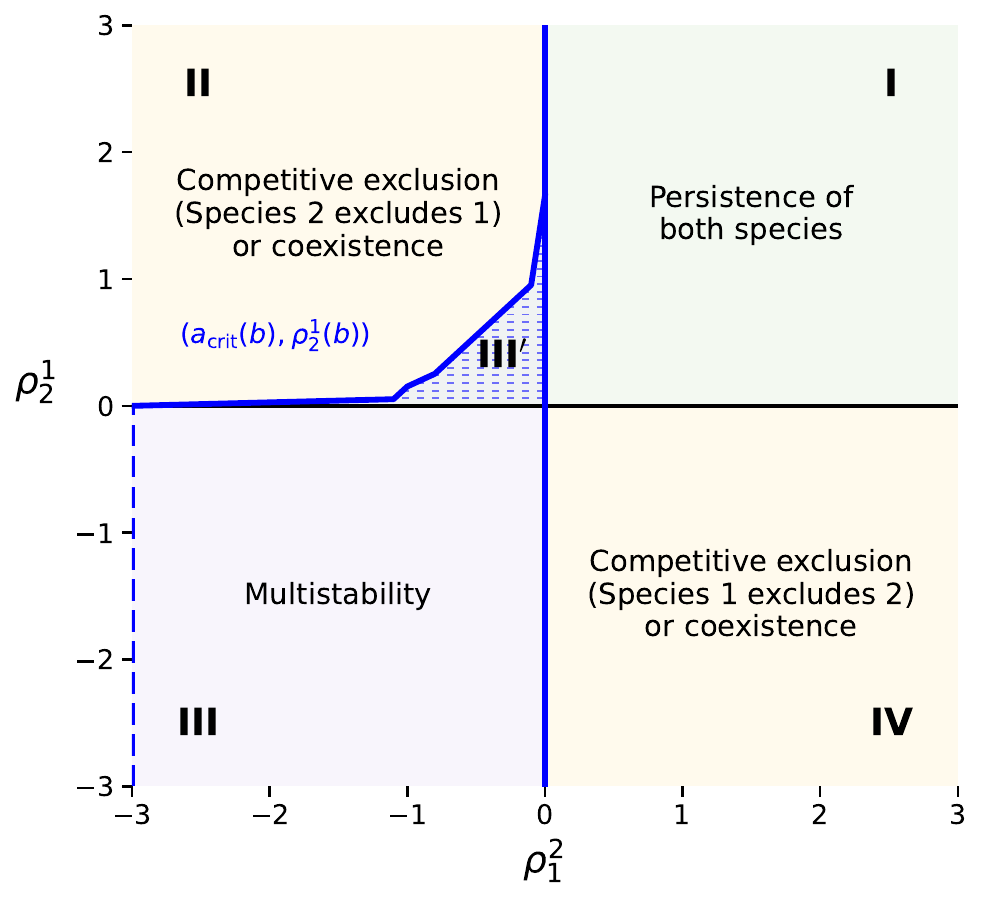}
    \caption{Parametric plot of \(\big(a_{\mathrm{crit}}(b),\rho_2^1(b)\big)\) in the \((\rho_1^2,\rho_2^1)\) plane for \(\lambda_1^2(x)=a\) and \(\lambda_2^1(x)=b-2x\), with \(D=0.01\) and \(\vec{\nu}=0\). For each value of \(b\), \(a_{\mathrm{crit}}(b)\) denotes the minimum value of \(a\) for which coexistence of both species is possible. Since \(\lambda_1^2(x)=a\) is constant in space, we have \(\rho_1^2=a\). Region III' is the additional multistability region in quadrant II, where coexistence is possible due to a backward bifurcation.}
    \label{fig:planefinal2}
\end{figure}

\subsection{Connection with spatial SIS coinfection models}

The examples above were written directly at the level of the spatial replicator equation. We now briefly show that this structure can also arise from a more classical epidemiological model. In particular, the slow-diffusion reduction of the spatial SIS coinfection model studied in \cite{Le2023} leads to a reaction-advection-diffusion replicator equation for the relative abundance of each strain, where the effective interaction coefficients \(\lambda_i^j(x)\) are determined by the strain-specific SIS parameters.

Figure~\ref{fig:sis-derived-replicator} gives one such example. The coefficients are obtained from spatially varying transmission rates and co-colonization clearance rates in the underlying SIS model. For the resulting spatial replicator equation, the principal eigenvalues are
\[
\rho_1^2 \simeq -4.1\times 10^{-3},
\qquad
\rho_2^1 \simeq 5.3\times 10^{-4}.
\]
Thus, the system lies in Quadrant II of the \((\rho_1^2,\rho_2^1)\) plane. In the corresponding non-spatial model, this sign structure would lead to exclusion of strain~1. Nevertheless, the simulation in Figure~\ref{fig:sis-derived-replicator}F shows that a stable spatial steady state with species~1 can exist, for certain initial conditions. This illustrates that the phenomena described above are not only consequences of considering abstract spatial interaction functions, but can also appear in standard SIS-type population dynamics in heterogeneous space.

\begin{figure}[htbp]
    \centering
    \includegraphics[width=0.99\linewidth]{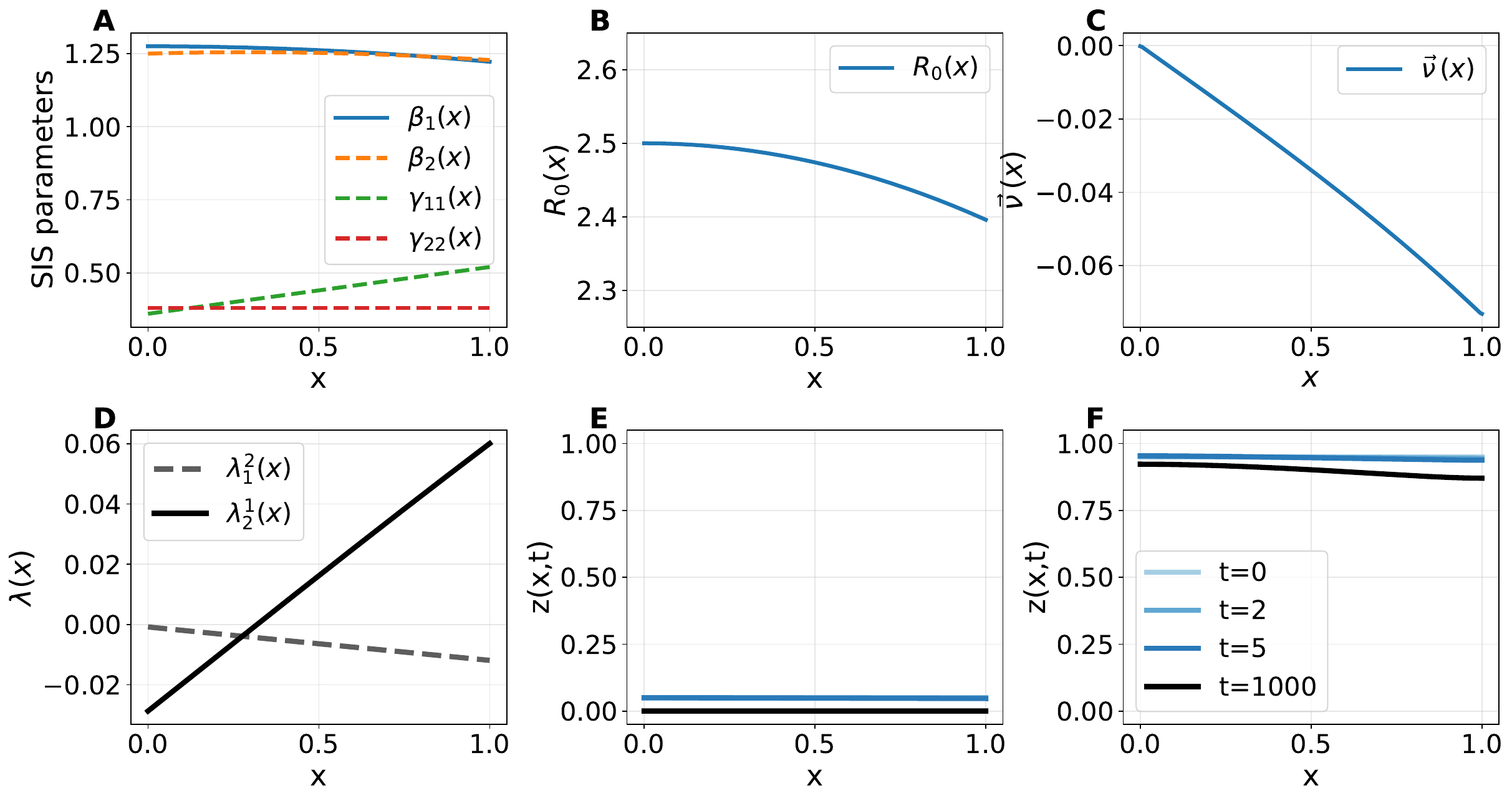}
    \caption{\textbf{A spatial replicator example generated from a classical SIS coinfection model.}
The interaction coefficients \(\lambda_i^j\) are obtained from the slow-diffusion reduction of a spatial SIS coinfection model. The diffusion term represents random spatial movement of hosts in the underlying SIS model, whereas the effective advection is generated by spatial gradients in the neutral endemic state, induced here by variation in \(R_0(x)\). For the resulting spatial replicator equation, the principal eigenvalues are \(\rho_1^2\simeq -4.1\times 10^{-3}\) and \(\rho_2^1\simeq 5.3\times 10^{-4}\). Thus, the system lies in Quadrant II of the \((\rho_1^2,\rho_2^1)\) plane, where the well mixed prediction would be exclusion of strain~1. Nevertheless, panel~F shows that such a stable spatial steady state can exist. We use \(r(x)=0.1\), \(\gamma(x)=0.4\), and \(k(x)=0.8\), and set \(R_0(x)=2.5-0.104x^2\), with \(\beta(x)=R_0(x)(r(x)+\gamma(x))\). Following the quasi-neutral parametrization, the strain-specific transmission and co-colonization clearance rates are defined by \(\beta_i(x)=\beta(x)\bigl(1+\epsilon b_i(x)\bigr)\), \(\gamma_{ii}(x)=\gamma(x)\bigl(1+\epsilon u_{ii}(x)\bigr)\). We take \(\epsilon=1\), \(b_1(x)=0.02\), \(b_2(x)=0.025x\), \(u_{11}(x)=0.4x-0.1\), and \(u_{22}(x)=-0.05\). All other perturbations are set to zero, so the only strain-specific parameters varying beyond the neutral background are the transmission rates \(\beta_1,\beta_2\) and the co-colonization clearance rates \(\gamma_{11},\gamma_{22}\).
(A) The resulting strain-specific SIS parameters.
(B) The basic reproduction number \(R_0(x)\), which remains above one throughout the domain.
(C) The advection coefficient \(v(x)\) induced by spatial variation in the neutral endemic state.
(D) The effective pairwise coefficients \(\lambda_1^2(x)\) and \(\lambda_2^1(x)\). In this example, \(\lambda_1^2(x)\) remains positive, whereas \(\lambda_2^1(x)\) changes sign in space.
(E,F) Time evolution of the corresponding spatial replicator equation from two different constant initial conditions, \(z(x,0) = 0.1, z(x,0) = 0.9\) showing that the same SIS-derived interaction structure can lead to different long-time outcomes. In F, the black profile at (t=1000) represents a stable spatial coexistence state.}
    \label{fig:sis-derived-replicator}
\end{figure}

\section{Discussion}

This work shows that spatial heterogeneity changes what can be inferred from invasion criteria. In the classical two-species replicator equation, the signs of the pairwise invasion fitnesses completely classify the dynamics \cite{Hofbauer1998,Cressman2014}. If both species can invade when rare, the coexistence equilibrium is stable; if only one can invade, the corresponding species excludes the other; and if neither can invade, the system is bistable. In the spatially heterogeneous replicator equation, the analogous quantities are the principal eigenvalues of the linearized operators around the boundary equilibria. These principal eigenvalues retain the ecological meaning of invasion fitnesses: they measure the initial growth rate of a rare species in an environment set by the resident species. However, our results show that their signs no longer provide a complete classification of the long-term dynamics. This separation between local invasion and longer-term outcome echoes a broader theme in adaptive dynamics: invasion fitness is a growth-when-rare criterion, but it may not by itself determine the subsequent evolutionary or ecological outcome \cite{Metz1992,Metz1996AdaptiveDynamics,Geritz1998}.

The spatial system can nevertheless support stable coexistence states even when one of the two spatial invasion fitnesses is negative. This creates a point of tension with invasion-based coexistence criteria, where invasion growth rates are central because they measure recovery from rarity and, under suitable assumptions, provide a test for persistence \cite{Chesson2000,Chesson2018}. In our setting, a negative principal eigenvalue still has its usual local meaning: the corresponding boundary equilibrium is stable against invasion by the rare species. What changes is the global implication of this fact. We show, through a backward bifurcation that there can be a stable interior state with its own basin of attraction, so failure to invade from rarity does not necessarily rule out coexistence. Thus, spatial heterogeneity separates two questions that are equivalent in the homogeneous two-species replicator equation: whether a species can invade when rare, and whether it can persist as part of a stable coexistence state.

The coexistence mechanism we obtain can be viewed as a form of spatial niche differentiation. In the bifurcation result, \(\lambda_1^2(x)>0\), while \(\lambda_2^1(x)\) varies sufficiently across space and changes sign, so the local competitive balance differs from one region to another. Diffusion and advection then couple these local environments, allowing spatial heterogeneity to affect the global dynamics. Importantly, this effect is not determined only by spatial averages of the interaction coefficients. The bifurcation condition involves weighted averages, with weights given by the principal eigenfunction of the linearized operator and its adjoint. Thus, the relevant quantity is not simply the amount of favorable habitat, but where favorable and unfavorable regions lie relative to dispersal and to the interaction profile of the other species. In this sense, the backward bifurcation provides a concrete mechanism by which spatially varying competitive rankings can generate coexistence \cite{Amarasekare2003,Cantrell1991,Dockery1998,Cosner2014}.

The effects of diffusion and advection reinforce this interpretation. Movement determines how the species experience the heterogeneous interaction landscape, and therefore how local competitive advantages are translated into a global outcome. In the diffusion examples, Figure~\ref{fig:effectsofdiffusion}, increasing \(D\) weakens the backward coexistence branch and eventually eliminates it, as the system approaches a more homogenized regime in which the boundary attractors dominate. This suggests that the coexistence mechanism relies on a decoupling between favorable and unfavorable regions, as too much diffusion averages out the spatial structure that supports the interior state. Advection changes this balance in a less straightforward way. By transporting frequencies preferentially across the domain, advection changes the spatial overlap between species distributions and favorable regions, and can therefore either create or destroy the backward bifurcation.

The backward bifurcation also gives initial conditions a more explicit role. In the same invasion-sign regime, the spatial system may contain both a stable boundary equilibrium and a stable coexistence state. The observed outcome can therefore depend on whether the initial distribution lies in the basin of attraction of the coexistence branch or in that of the exclusion state. A species with negative invasion fitness may fail to invade from rarity, yet persist if introduced above a sufficient threshold or in a favorable spatial configuration. Ecologically, this connects the model to ideas of priority effects and historical contingency in community assembly \cite{Fukami2015}. The final community is not determined only by the signs of the spatial invasion fitnesses, but also by how the system is assembled.

Beyond coexistence of the resident pair, our results also show that stability of a two-species spatial steady state does not have to be linked to resistance to further invasion. Along the stable coexistence branch, the resident community can be dynamically stable even when its mean invasion resistance  \(\overline Q_a\) is negative. In that case, the same spatial structure that supports coexistence between species~1 and species~2 makes the resident state very susceptible to a third species. Thus, stability within the resident pair and resistance to a new rare type are different aspects of the spatial steady state.

The SIS-derived example shows that this phenomenon is not restricted to interaction coefficients chosen directly at the level of the spatial replicator equation. Starting from the spatial SIS coinfection model of \cite{Le2023}, the slow-diffusion reduction leads to a reaction-advection-diffusion equation for strain frequencies, with effective interaction coefficients determined by the underlying epidemiological parameters. In the example considered here, the resulting system lies in a quadrant where the non-spatial prediction would be exclusion of one strain. Nevertheless, a stable spatial steady state with both strains can occur. This suggests that the backward-bifurcation mechanism can also arise in mechanistic biological models, with the spatial interaction terms emerging from the parameters of the underlying system.

There are also limitations to the present analysis. The main bifurcation result relies on the sign assumption \(\lambda_1^2(x)>0\), which allows the effect of the bifurcation parameter to act in the same way across space. This assumption is useful mathematically because it makes the comparison and bifurcation arguments tractable, but it does not cover all possible heterogeneous interaction structures. In particular, sign-changing \(\lambda_1^2(x)\) may generate additional forms of multistability or coexistence that are not explained by the mechanism identified here. Therefore, the results should not be read as a complete classification of spatial replicator dynamics. Rather, they identify one explicit and ecologically interpretable mechanism by which spatial heterogeneity can create stable coexistence outside the regimes predicted by the homogeneous two-species model.

Taken together, these results point to a broader difficulty in spatial ecological models. Once local interactions are coupled through movement in a heterogeneous environment, invasion criteria may remain biologically meaningful without fully determining the long-term dynamics. Coexistence and exclusion can also depend on arrival history, the spatial arrangement of favorable and unfavorable regions, and the way dispersal connects them. The two-species spatial replicator equation provides a useful minimal setting in which to separate these effects. By working with relative frequencies, the model reduces the dynamics to a single scalar equation, while still retaining the influence of spatial heterogeneity and dispersal. It therefore offers a tractable framework for studying how space modifies classical invasion and coexistence criteria.

\section*{Acknowledgements} 
This work was supported by Fundação para a Ciência e Tecnologia (FCT), Portugal, through project Models4Invasion (FCT grant 2022.03060.PTDC).

\section*{Conflict of interest} The authors declare no conflicts of interest.

\paragraph{SI Code.} 
\href{https://github.com/tomasfreire/Backward-bifurcations-in-spatial-replicator-models}{https://github.com/tomasfreire/Backward-bifurcations-in-spatial-replicator-models}

\printbibliography

\clearpage
\setcounter{figure}{0}
\setcounter{table}{0}
\setcounter{section}{0}
\setcounter{theorem}{0}

\renewcommand{\thesection}{S\arabic{section}}
\renewcommand{\thefigure}{S\arabic{figure}}
\renewcommand{\thetable}{S\arabic{table}}
\renewcommand{\thetheorem}{S\arabic{theorem}}

\renewcommand{\theHsection}{S\arabic{section}}
\renewcommand{\theHfigure}{S\arabic{figure}}
\renewcommand{\theHtable}{S\arabic{table}}
\renewcommand{\theHtheorem}{S\arabic{theorem}}

\section*{Supplementary Information (SI Text)}
\begin{large}
    \textbf{Backward bifurcations in spatial replicator models:
when invasion criteria fail to predict coexistence}\\

\noindent Tomás Freire$^{1}$,
Erida Gjini$^{1}$, Sten Madec$^{2}$\\

\noindent $^{1}${Center for Stochastic and Computational Mathematics, Instituto Superior Tecnico, University of Lisbon, Lisbon, Portugal}

\noindent $^{2}${Institut Denis Poisson, University of Tours, Tours, France}\\

\noindent  tomas.freire@tecnico.ulisboa.pt, erida.gjini@tecnico.ulisboa.pt, sten.madec@univ-tours.fr

\end{large}

\section{Useful Theorems and Lemmas}

\begin{theorem}[Theorem 3.7 \cite{Cantrell2003}]

\label{thm:rabinowitzSI}
    Suppose that $L$ and $H$ are operators as described above. If $\mu_0$ is a simple characteristic value of $L$, then $a=\mu_0$ is a bifurcation point for \eqref{eq:352}. Let $\mathcal{S} \subset \mathbb{R} \times Y$ be the set of nontrivial solutions to \eqref{eq:352}. Let $\mathcal{C}$ be the connected component of the set $\mathcal{S} \cup \{(\mu_0,0)\}$. Then, in a neighborhood of $(\mu_0,0)$, $\mathcal{C} = \mathcal{C}^+ \cup \mathcal{C}^-$, where $\mathcal{C}^+ \cap \mathcal{C}^- = \{(\mu_0,0)\}$ and each of $\mathcal{C}^+, \mathcal{C}^-$ satisfies one of the alternatives
    \begin{enumerate}
        \item $\mathcal{C}^+$ (resp. $\mathcal{C}^-$) is unbounded in $\mathbb{R}\times Y$, or
        \item $\mathcal{C}^+$ (resp. $\mathcal{C}^-$) contains $(\mu_1,0)$ where $\mu_1$ is a characteristic value of $L$ and $\mu_1 \neq \mu_0$.
    \end{enumerate}
\end{theorem}

\section{Derivation of the local bifurcation branch for frequency-dependent diffusion}
\label{SI:bifurcationdz}

In the main text, the local bifurcation calculation is carried out under the
assumption that \(D=D(x)\). In this case, the diffusion operator is linear in
\(z\), and the only nonlinear terms in the fixed-point formulation come from
the reaction term. Here we show how the calculation and assumptions change if the diffusion
coefficient also depends on the density, \(D=D(x,z)\).

The steady-state equation is then
\begin{equation}
\label{eq:SI_steady_zdependent_D}
0
=
z(1-z)\Big(a\lambda_1^2(x)-\big(a\lambda_1^2(x)+\lambda_2^1(x)\big)z\Big)
+\vec{\nu}(x)\cdot\nabla z
+\nabla\cdot\!\big(D(x,z)\nabla z\big),
\end{equation}
with homogeneous Neumann boundary conditions. Linearizing at the trivial
state \(z=0\) gives
\begin{equation}
\label{eq:SI_La_zdependent_D}
\mathcal L_a\psi
=
\nabla\cdot\!\big(D(x,0)\nabla\psi\big)
+
\vec\nu(x)\cdot\nabla\psi
+
a\lambda_1^2(x)\psi .
\end{equation}
Indeed, the dependence of \(D\) on \(z\) does not contribute to the first
linearization, because the term involving \(D_z(x,0)\) is multiplied by
\(\nabla z\) and is therefore of second order along the trivial branch.
Thus, a rigorous formulation for \(D=D(x,z)\) requires a space with enough
regularity to control the nonlinear diffusion term as a
\(C^{0,\alpha}(\overline\Omega)\)-valued map. A natural choice, considered in \cite{Cantrell2003}, for example, is
\[
Y=C^{2,\alpha}_N(\overline\Omega)
:=
\{z\in C^{2,\alpha}(\overline\Omega):\partial_{\vec n}z=0
\text{ on }\partial\Omega\}.
\] 
In this setting, \(z\), \(\nabla z\), and \(\Delta z\) are controlled in
Hölder norms. Under the corresponding Schauder assumptions on the coefficients
and the boundary, the shifted elliptic operator satisfies
\[
A^{-1}:C^{0,\alpha}(\overline\Omega)\to C^{2,\alpha}_N(\overline\Omega).
\]

The nonlinear diffusion contribution can be written as
\[
\begin{aligned}
\nabla\cdot\!\big((D(x,z)-D(x,0))\nabla z\big)
={}&
(D(x,z)-D(x,0))\Delta z \\
&+
\big(\nabla_xD(x,z)-\nabla_xD(x,0)\big)\cdot\nabla z \\
&+
D_z(x,z)|\nabla z|^2 .
\end{aligned}
\]
Under suitable regularity assumptions on \(D\), for instance enough Hölder
regularity in \(x\) and enough smoothness in \(z\) to make the nonlinear map
twice differentiable, this term belongs to \(C^{0,\alpha}(\overline\Omega)\).
Moreover, for \(z\) small,
\[
D(x,z)-D(x,0)=D_z(x,0)z+\mathcal O(z^2),
\]
and, under the corresponding mixed regularity assumption,
\[
\nabla_xD(x,z)-\nabla_xD(x,0)
=
\nabla_xD_z(x,0)z+\mathcal O(z^2).
\]
Hence
\[
\left\|
\nabla\cdot\!\big((D(x,z)-D(x,0))\nabla z\big)
\right\|_{C^{0,\alpha}}
\leq
C\|z\|_{C^{2,\alpha}}^2.
\]
After applying \(A^{-1}\), this gives
\[
\left\|
A^{-1}
\nabla\cdot\!\big((D(x,z)-D(x,0))\nabla z\big)
\right\|_{C^{2,\alpha}}
\leq
C\|z\|_{C^{2,\alpha}}^2.
\]
Together with the reaction term,
\[
\frac{\|H(a,z)\|_{C^{2,\alpha}}}{\|z\|_{C^{2,\alpha}}}
\longrightarrow 0
\qquad
\text{as } \|z\|_{C^{2,\alpha}}\to 0.
\]

We now compute the second-order contribution of the frequency-dependent diffusion
term to the branch direction. Let
\[
a=a(s),\qquad z=z(s),
\]
be the local bifurcating branch, with
\[
a(0)=a_0,\qquad z(0)=0,\qquad z_s(0)=\psi_0,
\qquad z_{ss}(0)=2\phi_s(0).
\]
Substituting this parametrization into \eqref{eq:SI_steady_zdependent_D} and
differentiating with respect to \(s\) gives
\begin{equation}
\label{eq:SI_first_s_derivative_zdependent_D}
\nabla \cdot\!\Big(D_z(x,z)z_s\nabla z\Big)
+
\nabla \cdot\!\Big(D(x,z)\nabla z_s\Big)
+
\vec\nu\cdot\nabla z_s
+
(zg)_s
=0 .
\end{equation}
Differentiating once more yields
\begin{equation}
\label{eq:SI_second_s_derivative_zdependent_D}
\begin{aligned}
&
\nabla \cdot\!\Big(D_{zz}(x,z)z_s^2\nabla z
+
D_z(x,z)z_{ss}\nabla z\Big)
+
2\nabla \cdot\!\Big(D_z(x,z)z_s\nabla z_s\Big) \\
&\qquad
+
\nabla \cdot\!\Big(D(x,z)\nabla z_{ss}\Big)
+
\vec\nu\cdot\nabla z_{ss}
+
(zg)_{ss}
=0 .
\end{aligned}
\end{equation}
Setting \(s=0\), using \(z(0)=0\), \(z_s(0)=\psi_0\), and
\(z_{ss}(0)=2\phi_s(0)\), gives
\begin{equation}
\label{eq:SI_second_s_zero_zdependent_D}
\begin{aligned}
&
2\nabla \cdot\!\Big(D_z(x,0)\psi_0\nabla\psi_0\Big)
+
\nabla \cdot\!\Big(D(x,0)\nabla z_{ss}(0)\Big)
+
\vec\nu\cdot\nabla z_{ss}(0) \\
&\qquad
+
a_0\lambda_1^2 z_{ss}(0)
+
2a_s(0)\lambda_1^2\psi_0
-
\left(4a_0\lambda_1^2+2\lambda_2^1\right)\psi_0^2
=0 .
\end{aligned}
\end{equation}

Let \(\psi_0^*>0\) be the corresponding principal eigenfunction of the
adjoint operator associated with \(\mathcal L_{a_0}\). Multiplying
\eqref{eq:SI_second_s_zero_zdependent_D} by \(\psi_0^*\), integrating over
\(\Omega\), and using the adjoint equation to remove the terms involving
\(z_{ss}(0)\), we obtain
\begin{equation}
\label{eq:SI_as_identity_zdependent_D}
2a_s(0)
\int_\Omega \lambda_1^2\psi_0\psi_0^*\,dx
=
2\int_\Omega
D_z(x,0)\psi_0\nabla\psi_0\cdot\nabla\psi_0^*\,dx
+
\int_\Omega
\left(4a_0\lambda_1^2+2\lambda_2^1\right)
\psi_0^2\psi_0^*\,dx .
\end{equation}
Therefore
\begin{equation}
\label{eq:SI_as_formula_zdependent_D}
a_s(0)
=
\frac{
\displaystyle
\int_\Omega
D_z(x,0)\psi_0\nabla\psi_0\cdot\nabla\psi_0^*\,dx
+
\int_\Omega
\left(2a_0\lambda_1^2+\lambda_2^1\right)
\psi_0^2\psi_0^*\,dx
}{
\displaystyle
\int_\Omega
\lambda_1^2\psi_0\psi_0^*\,dx
}.
\end{equation}
Thus, the density dependence of $D$ appears in the final formula and may change the branch direction.
\end{document}